\documentclass[12pt]{amsart}
\usepackage{amssymb}
\usepackage{bbm}
\usepackage{amsfonts}
\usepackage{latexsym}
\usepackage[all]{xy}
\usepackage{amscd}
\usepackage{comment}
\usepackage{fullpage}
\usepackage{pstricks}
\usepackage[bookmarks=false, colorlinks, linkcolor=mylinkcolor, 
citecolor=mycitecolor]{hyperref}
\definecolor{mycitecolor}{rgb}{0,.6,0} 
\definecolor{mylinkcolor}{rgb}{0,0,.8}   

\numberwithin{equation}{section}
\newtheorem{theorem}[equation]{Theorem}
\newtheorem{lemma}[equation]{Lemma}

\newtheorem{corollary}[equation]{Corollary}
\theoremstyle{definition}

\newtheorem{example}[equation]{Example}

\newtheorem{remark}[equation]{Remark}

\newcommand{\DOT}{\setlength{\unitlength}{1pt}\begin{picture}(2.5,2)
                  (1,1)\put(2,3.5){\circle*{2}}\end{picture}}

\newcommand{\bu}{\DOT}
\newcommand{\la}{\langle}
\newcommand{\ra}{\rangle}
\newcommand{\ot}{\otimes}
\newcommand{\Wedge}{\textstyle\bigwedge}
\newcommand{\lexp}[2]{{\vphantom{#2}}^{#1}{#2}}

\newcommand{\N}{\mathbb{N}}

\newcommand{\Z}{\mathbb{Z}}
\renewcommand{\O}{\{0,1\}}

\newcommand{\C}{\mathcal{C}}
\newcommand{\minusone}{\bf -\underline{1}}
\newcommand{\CC}{\mathbb{C}}
\newcommand{\R}{\mathcal{R}}
\newcommand{\ds}{\displaystyle}
\newcommand{\mH}{\mathcal{H}}

\DeclareMathOperator{\ddet}{det}
\DeclareMathOperator{\Ker}{Ker}
\DeclareMathOperator{\Hom}{Hom}
 
\DeclareMathOperator{\HH}{HH}
\DeclareMathOperator{\Ext}{Ext}

\DeclareMathOperator{\GL}{GL}

\renewcommand{\Im}{\operatorname{Im}}
\renewcommand{\span}{\operatorname{span}}

\begin{document}

\title[Quantum Drinfeld Hecke algebras]
{Hochschild cohomology and \\
quantum Drinfeld Hecke algebras}
\date{January 05, 2014} 
\subjclass[2010]{16E40, 16S35}

\author{Deepak Naidu}
\email{dnaidu@math.niu.edu}
\address{Department of Mathematical Sciences, Northern Illinois 
University, DeKalb, Illinois 60115, USA}
\author{Sarah Witherspoon}
\email{sjw@math.tamu.edu}
\address{Department of Mathematics, Texas A\&M University,
College Station, Texas 77843, USA}
\thanks{The second  author was partially supported by
NSF grant \#DMS-0800832 and Advanced Research Program Grant 
010366-0046-2007
from the Texas Higher Education Coordinating Board.}

\begin{abstract}
Quantum Drinfeld Hecke algebras are generalizations of Drinfeld Hecke 
algebras in which
polynomial rings are replaced by  quantum polynomial rings. 
We identify these  algebras as deformations of skew group algebras, giving 
an
explicit  connection to Hochschild cohomology.
We compute the relevant part of Hochschild cohomology for actions of 
many reflection groups and we exploit computations from \cite{NSW} 
for diagonal actions. By combining our work with recent results of 
Levandovskyy and
Shepler \cite{LS}, we produce examples of quantum Drinfeld Hecke algebras.
These algebras  generalize the braided Cherednik algebras of 
Bazlov and Berenstein \cite{BB}. 
\end{abstract}

\maketitle

\begin{section}{Introduction}

Let $G$ be a finite group acting linearly on a complex vector space $V$,
let $S(V)$ be the symmetric algebra on $V$, 
and let $S(V)\rtimes G$ be the corresponding skew group algebra (defined 
in the next
section).
Drinfeld (or graded) Hecke algebras manifest themselves
as deformations of these skew group algebras.
These deformations go by many other names, such as symplectic reflection 
algebras or 
rational Cherednik algebras, and have arisen in such diverse areas
as representation theory, combinatorics, and orbifold theory
\cite{C,D,EG,G,L,SW1}. 

In this note we
replace the symmetric algebra $S(V)$ with a quantum or twisted version:
Let 
$$
 S_{\bf q}(V) := \CC\la v_1,\ldots,v_n \mid v_iv_j = q_{ij}v_jv_i  \mbox{ 
for all }
      1\leq i,j\leq n \ra,
$$
the {\bf quantum symmetric algebra} determined by a basis $v_1,\ldots,v_n$ 
of $V$ and
a tuple ${\bf q}:= (q_{ij})$
of nonzero scalars for which $q_{ii}=1$ and $q_{ji}=q_{ij}^{-1}$ for all 
$i,j$.
The possible actions of the finite group $G$ on $S_{\bf q}(V)$ by linear 
automorphisms
are somewhat limited. Alev and Chamarie \cite{AC}
gave some results, but not a complete classification, of such actions.
Kirkman, Kuzmanovich, and Zhang \cite{KKZ} described actions of some
generalized reflection groups and proved a quantum version of
the classical Shephard-Todd-Chevalley Theorem; one consequence is that 
invariants of $S_{\bf q}(V)$ under these actions again form quantum 
symmetric
algebras. 

Bazlov and Berenstein \cite{BB} explored analogs of Cherednik algebras
in this context, termed  braided Cherednik algebras.
More generally:
Let $\kappa: V\times V\rightarrow \CC G$ be a bilinear map for which 
$\kappa(v_i,v_j) = - q_{ij}\kappa(v_j,v_i)$. 
Let $T(V)$ be the tensor algebra on $V$, 
and let 
$$
  \mH_{\mathbf{q}, \kappa} := T(V)\rtimes G/(  v_iv_j - q_{ij} v_jv_i - 
\kappa(v_i,v_j) \mid
     1\leq i,j\leq n),
$$
the quotient of the skew group algebra 
$T(V)\rtimes G$ by the ideal generated by all elements of the form 
specified.
Giving each $v_i$ degree 1 and each group element $g$  degree 0, 
$\mH_{\mathbf{q}, \kappa}$ is a
filtered algebra. We call $\mH_{\mathbf{q}, \kappa}$ a {\bf quantum 
Drinfeld Hecke algebra} if its
associated graded algebra is isomorphic to $S_{\bf q}(V)\rtimes G $.
In case all $q_{ij}=1$, these are the Drinfeld (or graded) Hecke algebras. 
Levandovskyy and Shepler \cite{LS} gave necessary and sufficient
conditions on the functions
$\kappa$  for $\mH_{\mathbf{q}, \kappa}$ to be a quantum Drinfeld Hecke 
algebra.

In this paper, we view these quantum  analogs of Drinfeld Hecke algebras
as deformations of the skew group algebras 
$S_{\bf q}(V)\rtimes G$. We establish the following theorem, which makes 
explicit a connection to 
Hochschild cohomology, thus forging another path to understanding these 
and related deformations.
The notation and terminology used below is explained in Section~\ref{QDHAs 
and deformations}.

\medskip
\noindent
{\bf Theorem \ref{main-thm}.} \textit{
The quantum Drinfeld Hecke algebras over $\CC[t]$ are precisely the 
deformations of
$S_{\bf q}(V)\rtimes G$ over $\CC[t]$ with $\deg\mu_i = -2i$ for all 
$i\geq 1$.
}
\medskip

With Shroff \cite{NSW} we computed the Hochschild cohomology of
$S_{\bf q}(V)\rtimes G$ in case $G$ acts {\em diagonally} on the chosen 
basis 
$\{v_i\}_{1\leq i\leq n}$ of $V$.
Here we consider more general actions, focusing on that
part of Hochschild cohomology in degree 2 that is relevant to quantum
Drinfeld Hecke algebras. We apply the criteria of Levandovskyy and Shepler 
\cite{LS} to show that when the action of $G$ extends to an action on a 
quantum exterior algebra, 
all such Hochschild 2-cocycles do indeed give rise to quantum Drinfeld 
Hecke algebras:

\medskip
\noindent
{\bf Theorem \ref{thm: constant}.} \textit{
Assume that the action of $G$ on $V$ extends to an action on $\Wedge_{\bf 
q}(V)$
by algebra automorphisms. Then each constant Hochschild 2-cocycle on 
$S_{\bf q}(V)\rtimes G$
gives rise to a quantum Drinfeld Hecke algebra.
}
\medskip

Combining the previous two theorems, we obtain the following.

\medskip
\noindent
{\bf Theorem \ref{thm: constant lifts}.} \textit{
Assume that the action of $G$ on $V$ extends to an action on $\Wedge_{\bf 
q}(V)$
by algebra automorphisms. Then each constant Hochschild 2-cocycle on 
$S_{\bf q}(V)\rtimes G$
lifts to a deformation of $S_{\bf q}(V)\rtimes G$ over $\CC[t]$. 
}
\medskip

We compute the relevant part of the Hochschild cohomology of $S_{\bf 
q}(V)\rtimes G$ in degree 2 for 
several types of complex reflection groups and actions. 
Our deformations include all of the braided Cherednik algebras of Bazlov
and Berenstein \cite{BB}, putting them in a larger context: 
Their vector space $V$ is always the direct sum of a vector space
and its dual, and they have some mild additional restrictions on 
the structure of the corresponding deformations of $S_{\bf q}(V)\rtimes G$.

We will work over the complex numbers $\CC$, and all tensor products will 
be
taken over $\CC$ unless otherwise indicated.

\vspace{0.1in}

\noindent {\bf Organization.} This paper is organized as follows. 

We prove Theorem~\ref{main-thm} in Section~\ref{QDHAs and deformations}.
Section~\ref{two resolutions} develops the homological algebra needed for 
Section \ref{LS-Theorem},
in which we will obtain results on the Hochschild 2-cocycles associated to 
quantum
Drinfeld Hecke algebras. In particular, we prove Theorem~\ref{thm: 
constant}.

In Section~\ref{diagonal}, using results from \cite{NSW} we classify 
quantum Drinfeld Hecke algebras 
for diagonal actions of $G$ on a chosen basis for $V$. 

In Sections~\ref{natural} and \ref{symplectic}, we consider the natural 
and symplectic representations of several types of 
complex reflection groups. In each case we classify the corresponding 
quantum Drinfeld Hecke algebras by computing the
relevant part of the Hochschild cohomology of $S_{\bf q}(V)\rtimes G$ in 
degree 2.

\end{section}
\begin{section}{Quantum Drinfeld Hecke algebras and deformations of 
$S_{\bf q}(V)\rtimes G$} \label{QDHAs and deformations}

Assume the finite group $G$ acts linearly on the complex vector space $V$, 
and
that there is an induced action on $S_{\bf q}(V)$ by algebra 
automorphisms. 
Then we may form the {\bf skew group algebra} $S_{\bf q}(V)\rtimes G$, and 
we
recall its definition:
Letting $A=S_{\bf q}(V)$, additively $A\rtimes G$ 
is the free left $A$-module with basis $G$.
We write  $A\rtimes G = \oplus_{g\in G}A_g$,
where $A_g = \{a g\mid a\in A\}$, that is for each $a\in A$ and
$g\in G$ we denote by $a g\in A_g$ the $a$-multiple of~$g$.
Multiplication on $A\rtimes G$ is determined by 
$$
   (a g) (b  h) := a (\lexp{g}{b})  gh
$$
for all $a,b\in A$ and $g,h\in G$, where a left superscript denotes
the action of the group element. 
Similarly we define the skew group 
algebra for any other algebra on which $G$ acts by automorphisms,
such as the tensor algebra $T(V)$.

We will show how quantum Drinfeld Hecke algebras may be realized as
deformations of $S_{\bf q}(V)\rtimes G$ by extending the scalars to 
$\CC[t]$:
For any algebra $R$ over $\CC$, a {\bf deformation of $R$ over $\CC[t]$}
is an associative $\CC[t]$-algebra whose underlying vector space is
$R[t] = \CC[t]\ot R$, and multiplication is
$$
  r*s = rs + \mu_1(r\ot s) t + \mu_2(r\ot s)t^2 + \cdots
$$
for all $r,s\in R$, where $rs$ is the product in $R$, the
$\mu_i :R\ot R \rightarrow R$ are $\CC$-linear maps extended to
be linear over $\CC[t]$, and for each $r,s$ the above sum is finite.
One consequence of associativity is that $\mu_1$ is a {\bf Hochschild 
2-cocycle},
that is 
\begin{equation}\label{eqn:2-cocycle-condn}
  \mu_1(r\ot s) u + \mu_1(rs\ot u) = \mu_1(r\ot su) + r\mu_1(s\ot u)
\end{equation}
for all $r,s,u\in R$. 

Let $\kappa: V\times V\rightarrow \CC G$ be a function as specified in the 
introduction. 
For each $g\in G$, let $\kappa_g: V\times V\rightarrow \CC$ be the 
function determined by the
condition 
\[
\kappa(v,w) = \sum_{g \in G} \kappa_g(v,w)g \qquad \text{ for all }  v, w 
\in V.
\]
(The condition $\kappa(v_i,v_j) = -q_{ij}\kappa(v_j,v_i)$ implies that 
$\kappa_g(v_i,v_j) = -q_{ij}\kappa_g(v_j,v_i)$ for each $g \in G$.
It arises when interchanging $i$ and $j$ in the defining relations of 
$\mH_{\mathbf{q}, \kappa}$; in the absence of this condition,
$\mH_{\mathbf{q}, \kappa}$ is too small in the sense that the group $G$ 
does not
embed in $\mH_{\mathbf{q}, \kappa}$ as a subgroup of its group of units.)
Let 
$$
  \mH_{\mathbf{q}, \kappa, t}:= T(V)\rtimes G[t]/ ( v_iv_j - q_{ij} v_jv_i 
- \sum_{g\in G} \kappa_g(v_i,v_j) t g
    \mid 1\leq i,j\leq n).
$$
Giving each $v_i$
degree 1 and each $g\in G$ and $t$ degree 0, we see that $\mH_{\mathbf{q}, 
\kappa, t}$ is a
filtered algebra. We are interested in those algebras $\mH_{\mathbf{q}, 
\kappa, t}$ for which the associated
graded algebra is isomorphic to $S_{\bf q}(V)\rtimes G [t]$; call these 
algebras
{\bf quantum Drinfeld Hecke algebras over $\CC[t]$}.
Specializing to $t=1$, these are the quantum Drinfeld Hecke algebras as
defined in the introduction.

We next prove that quantum Drinfeld Hecke algebras over $\CC[t]$
are all of the deformations of $S_{\bf q}(V)\rtimes G$ of a particular 
form.
The proof of Theorem \ref{main-thm} below is a straightforward 
generalization
of a special case 
of \cite[Theorem 3.2]{Wi}. We include a proof as we will need some of the 
details
and wish 
to highlight the homological meaning of the quantum skew-symmetry of the 
functions $\kappa_g$.

The {\bf quantum exterior algebra} $\Wedge _{\bf q}(V)$ associated
to the tuple ${\bf q} = (q_{ij})$ is
$$
  \Wedge_{\bf q}(V):= \CC\langle v_1,\ldots,v_n\mid v_iv_j = -q_{ij}v_jv_i
   \mbox{ for all } 1\leq i,j\leq n\rangle.
$$
Since we are working in characteristic 0, 
the defining relations imply in particular that 
$v_i^2=0$ for each $v_i$ in $\Wedge_{\bf q}(V)$.
This algebra has a basis given by all $v_{i_1}\cdots v_{i_m}$ ($0\leq 
m\leq n$,
$1\leq i_1<\cdots <i_m\leq n$); we will write such a basis element as 
$v_{i_1}\wedge
\cdots \wedge v_{i_m}$ by analogy with the ordinary exterior algebra.

In the theorem below, by the {\bf degree} of $\mu_i$, we mean its degree 
as a function
from the graded algebra $(A\rtimes G)^{\ot 2}$ to $A\rtimes G$. 
The theorem gives a one-to-one correspondence between quantum Drinfeld 
Hecke algebras
over $\CC[t]$ and deformations of $S_{\bf q}(V)\rtimes G$ over $\CC[t]$ 
satisfying a
condition on the degrees of the functions $\mu_i$.

\begin{theorem}\label{main-thm}
The quantum Drinfeld Hecke algebras over $\CC[t]$ are precisely the 
deformations of
$S_{\bf q}(V)\rtimes G$ over $\CC[t]$ with $\deg\mu_i = -2i$ for all 
$i\geq 1$.
\end{theorem}

\begin{proof}
Let $\mH_{\mathbf{q}, \kappa, t}$ be a quantum Drinfeld Hecke algebra over 
$\CC[t]$.
By its definition, this implies that $\mH_{\mathbf{q}, \kappa, t}$ is a 
deformation of 
$S_{\bf q}(V)\rtimes G$ over $\CC[t]$.
Specifically, since the associated graded algebra of $\mH_{\mathbf{q}, 
\kappa, t}$ is isomorphic
to $S_{\bf q}(V)\rtimes G[t]$, and $S_{\bf q}(V)$ has a basis
consisting of all monomials $v_1^{i_1}\cdots v_n^{i_n}$, each element of
$\mH_{\mathbf{q}, \kappa, t}$ may be written uniquely as a $\CC[t]$-linear 
combination of elements
of the form $v_1^{i_1}\cdots v_n^{i_n} g$. Let $r=v_1^{i_1}\cdots 
v_n^{i_n}g$ and
$s=v_1^{j_1}\cdots v_n^{j_n} h$ be two such basis elements of 
$\mH_{\mathbf{q}, \kappa, t}$.
Denoting the product in $\mH_{\mathbf{q}, \kappa, t}$ by $*$,
since $\mH_{\mathbf{q}, \kappa, t}$ is defined as a quotient of 
$T(V)\rtimes G[t]$, we have 
$$
   r*s = v_1^{i_1}\cdots v_n^{i_n} * 
\lexp{g}{(v_1^{j_1}\cdots v_n^{j_n})}gh,
$$
and applying the relations defining $\mH_{\mathbf{q}, \kappa, t}$ 
repeatedly, we obtain an
expression of the form
$$
   r*s = rs + \mu_1(r\ot s)t+\mu_2(r\ot s)t^2 + \cdots .
$$
The sum will be finite since each time a relation is applied, the degree 
drops.
The product $*$ makes $\mH_{\mathbf{q}, \kappa, t}$ an associative algebra 
by definition, and consequently the functions
$\mu_i$ will be bilinear. Therefore $\mH_{\mathbf{q}, \kappa, t}$ is a 
deformation of 
$S_{{\bf q}}(V)\rtimes G$ over $\CC[t]$. The conditions on the degrees of
the $\mu_i$ follow from the relations and induction on the degree 
$\sum_{k=1}^n (i_k+j_k)$ of a product $v_1^{i_1}\cdots v_n^{i_n} * 
v_1^{j_1}\cdots
v_n^{j_n}$.

Conversely, suppose that $B$ is a deformation of $S_{\bf q}(V)\rtimes G$
over $\CC[t]$ satisfying the given degree conditions.
Then, as a vector space over $\CC[t]$, $B \cong S_{\bf q}(V)\rtimes G[t]$.
Define a map $\phi: T(V)\rtimes G[t]\rightarrow B$ by first requiring
$$
   \phi(v_i) = v_i \ \ \ \mbox{ and } \ \ \ \phi(g)=g
$$
for all $i$, $1\leq i\leq n$, and $g\in G$.
Since $T(V)$ is free on a basis of $V$, the map $\phi$ on the $v_i$ may be
extended uniquely to an algebra homomorphism from $T(V)$ to $B$.
By the degree condition on the $\mu_i$, we have $\mu_i(\CC G, \CC G) =
\mu_i(\CC G, V) = \mu_i(V, \CC G) =0$, so $\phi$ may be extended to a 
$\CC[t]$-algebra
homomorphism on all of $T(V)\rtimes G[t]$, as desired.
Specifically,
$\phi(v_{i_1}\cdots v_{i_m} g ) = v_{i_1} * \cdots * v_{i_m} * g$. By the 
degree
requirements, we have for example $ g*v_{i_1}*v_{i_2} = 
(g*v_{i_1})*v_{i_2} = \lexp{g}{v_{i_1}} *g
*v_{i_2} = \lexp{g}{v_{i_1}} * \lexp{g}{v_{i_2}} * g$. 

We claim that $\phi$ is surjective. We will prove that each basis element
is in the image of $\phi$ by induction on the degree of the basis monomial.
First note that $\phi(g)=g$ and $\phi(v_ig) = v_ig$ for all $i$, $1\leq 
i\leq n$,
and all $g\in G$.
Now let $v_{i_1}\cdots v_{i_m} g$ be an arbitrary basis monomial of $B$.
By induction, $v_{i_2}\cdots v_{i_m}g$ is in the image of $\phi$, say 
$\phi(X) = v_{i_2}\cdots v_{i_n} g$ for some $X\in T(V)\rtimes G[t]$.
Then
\begin{eqnarray*}
   \phi(v_{i_1}X ) & = & v_{i_1} * \phi(X) \\
  &=& v_{i_1} * (v_{i_2}\cdots v_{i_m}g)\\
  &=& v_{i_1}\cdots v_{i_m} g + \mu_1(v_{i_1}, v_{i_2}\cdots v_{i_m}g)t 
    + \mu_2(v_{i_1},v_{i_2}\cdots v_{i_m}g ) t^2 + \cdots
\end{eqnarray*}
By induction, since $\deg(\mu_i) = -2i$, each $\mu_j(v_{i_1} , 
v_{i_2}\cdots
v_{i_m}g)$ is in the image of $\phi$. 
Therefore $v_{i_1}\cdots v_{i_m}g$ is in the image of $\phi$,
which implies $\phi$ is surjective.

Finally we determine the kernel of $\phi$.
Note that
\begin{eqnarray*}
  \phi(v_iv_j) & = & v_i*v_j \ \ = \ \ v_iv_j + \mu_1(v_i\ot v_j) t\\
  \phi(v_jv_i) & = & v_j*v_i \ \ = \ \  v_jv_i + \mu_1(v_j\ot v_i)t
\end{eqnarray*}
since $\deg(\mu_i)=-2i$ for all $i$.
As $v_jv_i = q_{ji}v_iv_j$ in $S_{\bf q}(V)$, we find 
$$
  \phi(v_jv_i - q_{ji}v_iv_j) = (\mu_1(v_j\ot v_i)-q_{ji}\mu_1(v_i\ot 
v_j))t.
$$
Since $\deg(\mu_1) = -2$ and $\phi(g)=g$ for all $g\in G$, this implies 
that
\begin{equation}\label{eqn:I[t]-gens}
  v_jv_i - q_{ji} v_iv_j - (\mu_1(v_j\ot v_i) - q_{ji}\mu_1(v_i\ot v_j))t
\end{equation}
is in the kernel of $\phi$ for all $i,j$, and that
$$
  \mu_1(v_j\ot v_i)-q_{ji}\mu_1(v_i\ot v_j) =
   \sum_{g\in G} \kappa_g(v_j , v_i) g
$$
for some functions $\kappa_g$. By interchanging $i,j$ we find that
$\kappa_g(v_i , v_j) = -q_{ij}\kappa_g(v_j , v_i)$, and by definition each
$\kappa_g$ is linear, so we may view each $\kappa_g$ as a linear function 
on
$\Wedge^2_{\bf q}(V)$, or equivalently as a bilinear function on $V\times 
V$
that satisfies $\kappa_g(v_i,v_j)= - q_{ij}\kappa_g(v_j,v_i)$ for all 
$i,j$. 
Let $I[t]$ be the ideal of $T(V)\rtimes G[t]$ generated by all such 
expressions (\ref{eqn:I[t]-gens}), so that $I[t]\subset \Ker\phi$.
We claim that $I[t]=\Ker\phi$: By the form of the relations, 
as a vector space $T(V)\rtimes G[t]/I[t]$ is a quotient of $S_{\bf 
q}(V)\rtimes G[t]$, so
has dimension in each degree no greater than that of $S_{\bf q}(V)\rtimes 
G[t]$.
Since $\phi$ induces a map from $T(V)\rtimes G[t]/I[t]$ onto the vector 
space
$B \cong S_{\bf q}(V)\rtimes G[t]$,
this forces $I[t] =\Ker \phi$. Therefore $B$ is a quantum Drinfeld
Hecke algebra.
\end{proof}

In particular, if we search for quantum Drinfeld Hecke algebras,
Theorem \ref{main-thm} 
shows that we might first determine the Hochschild two-cocycles
$\mu_1$ of degree $-2$ as maps from $(A\rtimes G)^{\ot 2}$ to $A\rtimes G$.
We will call these {\bf constant} Hochschild 2-cocycles; this choice of 
terminology
will be justified by results of the next section,
where we recall and develop the needed tools from homological algebra.
As a consequence we will show in Theorem \ref{thm: constant}
that in fact 
{\em all} constant Hochschild 2-cocycles give rise to quantum Drinfeld
Hecke algebras under the condition that the action of $G$ extends to
an action on $\Wedge_{\bf q}(V)$ by algebra automorphisms. 

One outcome of the above proof is an explicit relationship between
the functions $\kappa_g$ and the Hochschild 2-cocycles $\mu_1$:
$$
   \sum_{g\in G} \kappa_g(v_j,v_i)g = \mu_1(v_j\ot v_i) - 
q_{ji}\mu_1(v_i\ot v_j).
$$

\end{section}

\begin{section}{Two resolutions} \label{two resolutions}

In this section we develop the homological algebra needed for Section 
\ref{LS-Theorem},
in which we will obtain results on the Hochschild 2-cocycles associated to 
quantum
Drinfeld Hecke algebras.

The Hochschild cohomology of an algebra $R$ is $\HH^*(R):= 
\Ext^*_{R^e}(R,R)$,
where the enveloping algebra $R^e :=R\ot R^{op}$ acts on $R$ by left and 
right multiplication. 
When $R=A\rtimes G$ is a skew group algebra in a characteristic not 
dividing the
order of the finite group $G$,
it is well-known that there is an  action of $G$ on $\HH^*(A,A\rtimes G):= 
\Ext^*_{A^e}(A,A\rtimes G)$ for which $\HH^*(A\rtimes G)\cong 
\HH^*(A,A\rtimes G)^G$,
the elements of $\HH^*(A, A\rtimes G)$ that are invariant under $G$.
(See, for example, \c{S}tefan \cite[Corollary 3.4]{S}.)

For the purpose of computing Hochschild cohomology, 
we first recall the quantum Koszul resolution.
Our goal is to understand the Hochschild
cohomology of $S_{\bf q}(V)$ and of $S_{\bf q}(V)\rtimes G$, and
to use this knowledge to give explicitly any corresponding deformations 
of $S_{\bf q}(V)\rtimes G$.

Set $A=S_{\bf q}(V)$.
For each $g\in G$, $A_g$ is a (left)
$A^e$-module via the action 
$$
   (a\ot b) \cdot (c g) := ac g b = ac (\lexp{g}{b})  g
$$
for all $a,b,c\in A$, $g\in G$.
According to  Wambst \cite[Proposition 4.1(c)]{W},
the following is a free $A^e$-resolution of $A$:

\begin{equation}
\label{label: free resolution}
\cdots \xrightarrow{} A^e\ot\Wedge_{\bf q}^2(V) \xrightarrow{d_2}
A^e \otimes \Wedge_{\bf q} ^1(V) \xrightarrow{d_1}
A^e \xrightarrow{\text{mult}} A \xrightarrow{} 0,
\end{equation}
that is, for $1\leq m\leq n$, the degree $m$ term is $A^e\otimes 
\Wedge_{\bf q}^m(V)$;
the differential $d_m$  is defined by 
\begin{equation*}
\begin{split}
&d_m(1^{\ot 2}\ot v_{j_1}\wedge\cdots\wedge v_{j_m})\\
&=  \sum_{i=1}^m (-1)^{i+1} \left[ \left(\prod_{s=1}^{i} q_{j_s, j_i} 
\right)
    v_{j_i}\ot 1 - \left(\prod_{s=i}^m q_{j_i, j_s} \right) \ot v_{j_i} 
\right] \ot
    v_{j_1}\wedge \cdots \wedge \hat{v}_{j_i} \wedge\cdots \wedge v_{j_m}
\end{split}
\end{equation*}
whenever $1\leq j_1 <\ldots < j_m\leq n$, and mult denotes the 
multiplication
map. 
(While $\Wedge^m_{\bf q}(V)$ is isomorphic to $\Wedge^m(V)$ as a vector 
space, we retain
the ${\bf q}$ in the notation as a reminder to apply the relation
$v_i\wedge v_j = - q_{ij}v_j\wedge v_i$ whenever we wish to rewrite 
elements in this way,
such as after having applied a group action.)
The complex (\ref{label: free resolution}) 
is a twisted version of the usual Koszul resolution for a polynomial ring.
See also Bergh and Oppermann \cite{BO} for construction of 
more general twisted products of resolutions.

Let us write the above formula for $d_m$ in a more convenient form. 
We first introduce some notation following  Wambst \cite{W}. 
Let $\N^n$ denote the set of all $n$-tuples of elements
from $\N$. For any $\alpha = (\alpha_1,\ldots,\alpha_n) \in \N^n$, the 
{\bf length}
of $\alpha$, denoted $|\alpha|$, is the sum $\sum_{i=1}^n \alpha_i$.
For all $\alpha \in \N^n$, define $v^\alpha := 
v_1^{\alpha_1} v_2^{\alpha_2} \cdots v_n^{\alpha_n}$. For all 
$i \in \{1, \ldots, n\}$,  define $[i] \in \N^n$ by
$[i]_j = \delta_{i,j}$, for all $j \in \{1, \ldots, n\}$.
For any $\beta=(\beta_1,\ldots,\beta_n) \in \O^n$, let $v^{\wedge \beta}$ 
denote
the vector $v_{j_1} \wedge \cdots \wedge v_{j_m} \in \Wedge_{\bf q}^m(V)$
which is defined by $m = |\beta|$, $\beta_{j_k} = 1$ for all 
$k \in \{1, \ldots, m\}$, and $j_1<\ldots <j_m$.
Then, for any $\beta \in \O^n$ with $|\beta| = m$ we have
\begin{equation*}
d_m(1^{\ot 2}\ot v^{\wedge \beta}) =
\sum_{i=1}^n \delta_{\beta_i,1} (-1)^{\sum_{s=1}^{i-1} \beta_s} 
\left[ \left(\prod_{s=1}^i q_{s, i}^{\beta_s} \right)
    v_i\ot 1 - \left(\prod_{s=i}^n q_{i, s}^{\beta_s} \right) \ot v_i 
\right] \ot
    v^{\wedge(\beta-[i])}.
\end{equation*}

Applying the functor $\Hom_{A^e}(\cdot, A_g)$ 
to the $A^e$-resolution of $A$ in \eqref{label: free resolution},
and making appropriate identifications, we obtain 
\begin{equation}
\label{new Hom(resolution) with G}
0 \xrightarrow{}  A_g \xrightarrow{d_1^*} 
A_g \ot \Wedge^1_{{\bf q}^{-1}}(V^*) \xrightarrow{d_2^*} 
A_g \ot \Wedge^2_{{\bf q}^{-1}}(V^*) \xrightarrow{} \cdots,
\end{equation}
where  $d_m^*(a g\ot (v^*)^{\wedge\beta})$ is equal to 
\begin{equation}
\label{formula for d_m^* with G}
\sum_{i=1}^n
\delta_{\beta_i,0} (-1)^{\sum_{s=1}^i \beta_s} \left[
\left( \left( \prod_{s=1}^i q_{s,i}^{\beta_s} \right) v_ia - 
\left( \prod_{s=i}^n q_{i,s}^{\beta_s}\right) a(\lexp{g}{v_i}) \right) g 
\right]
\ot {(v^*)}^{\wedge(\beta+[i])},
\end{equation}
for all $a \in A$ and $\beta \in \O^n$ with $|\beta|=m-1$.
(Note that the relations on dual functions are indeed $v_i^*\wedge v_j^* =
- q_{ij}^{-1}v_j^*\wedge v_i^*$, for all $i,j$, as may be determined by
applying each side of this equation to $v_i\wedge v_j$ and using the
defining relations in $\Wedge_{\bf q}(V)$.)

We may compute $\HH^m(A \rtimes G)$ as follows:
$$
   \HH^m(A,A_g) \cong \Ker d^*_{m+1}/\Im d^*_{m}
\qquad \text{and} \qquad
   \HH^m(A \rtimes G) \cong \left(\bigoplus_{g\in G} 
\HH^{m}(A,A_g)\right)^G.
$$

Hochschild 2-cocycles give rise to elements of
$\Hom_{\CC}((A \rtimes G)^{\ot 2}, A \rtimes G)$
that satisfy the $2$-cocycle condition (\ref{eqn:2-cocycle-condn}). 
We wish to describe this correspondence
explicitly. To this end, we will next introduce maps
that translate between the complex (\ref{new Hom(resolution) with G}) and 
the bar complex for $A\rtimes G$.

First we consider chain maps between the quantum Koszul resolution 
(\ref{label: free resolution})
and the bar resolution of $A$:
$$
\xymatrix{
\cdots \ar[r] & A^{\ot 4}\ar[r]^{\delta_2}\ar@<-2pt>[d]_{\Psi_2} 
               & A^{\ot 3}\ar[r]^{\delta_1}\ar@<-2pt>[d]_{\Psi_1} 
               & A^e \ar[r]^{\text{mult}} \ar[d]_{=}
               & A \ar[r] \ar[d]_{=} & 0\\
\cdots \ar[r] & A^e\ot \Wedge_{\bf q} ^2 V 
\ar[r]^{d_2}\ar@<-2pt>[u]_{\Phi_2}
               & A^e\ot \Wedge_{\bf q} ^1 V 
\ar[r]^{\hspace{.6cm}d_1}\ar@<-2pt>[u]_{\Phi_1}
               & A^e \ar[r]^{\text{mult}} \ar[u]
               & A\ar[r] \ar[u] & 0 .
}
$$
Here the differentials $\delta_i$ in the bar resolution are defined as 
$$
\delta_i(a_0\ot\cdots \ot a_{i+1}) = \sum_{j=0}^i (-1)^j a_0\ot\cdots \ot 
a_j a_{j+1}
  \ot\cdots\ot a_{i+1}
$$
for all $a_0,\ldots,a_{i+1}\in A$.
We will only need to know the values of $\Psi_2$ on elements of the form
$1\ot v_i\ot v_j\ot 1$, and to find these values, we will need to know the 
values
of $\Psi_1$ on $1\ot v_i\ot 1$ and on $1\ot v_iv_j\ot 1$.
Since the bar resolution consists of free modules, we may
choose these values to be any inverse images, under $d_1$, of 
$\delta_1(1\ot v_i\ot 1)$
and of $\delta_1(1\ot v_iv_j\ot 1)$.
We choose $\Psi_1(1\ot v_i\ot 1) = 1\ot 1\ot v_i$ and 
$\Psi_1(1\ot v_iv_j\ot 1) = q_{ij}\ot v_i\ot v_j + q_{ij}v_j\ot 1\ot v_i$.
It follows that 
\begin{eqnarray*}
  \Psi_1\delta_2(1\ot v_i\ot v_j\ot 1) &=& \Psi_1(v_i\ot v_j\ot 1 
   -1\ot v_iv_j\ot 1 + 1\ot v_i\ot v_j)\\
    &=& (v_i\ot 1 - q_{ij}\ot v_i)\ot v_j - (q_{ij}v_j\ot 1 - 1\ot v_j)\ot 
v_i.
\end{eqnarray*}
When $i<j$,
this is precisely $d_2(1\ot 1\ot v_i\wedge v_j)$, so we may let
\begin{equation}\label{psi-two}
  \Psi_2(1\ot v_i\ot v_j\ot 1) = 1\ot 1\ot v_i\wedge v_j \ \ \ (1\leq 
i<j\leq n).
\end{equation}
A similar analysis shows that we may let $\Psi_2(1\ot v_i\ot v_j\ot 1)=0$ 
whenever $i\geq j$.
(This asymmetric choice can make hand computations less onerous; 
alternatively, a more elegant,
quantum symmetric choice is possible.)


Chain maps $\Phi_i$ are defined in \cite{NSW}, and more generally in 
\cite{W}, 
that embed the quantum Koszul resolution
as a subcomplex of the bar resolution. We will not need these maps here.




Two more maps are defined as in \cite{SW2}:
We define the Reynold's operator, or averaging map, which ensures
$G$-invariance of the image, compensating for the possibility that 
$\Psi_2$ may not preserve the action of $G$:
\begin{eqnarray*}
\R_2: \Hom_{\CC}(A^{\ot 2}, A \rtimes G) &\to & \Hom_{\CC}(A^{\ot 2}, A 
\rtimes G)^G\\
\R_2(\gamma) &:=  &\frac{1}{|G|} \sum_{g \in G} \lexp{g}{\gamma}.
\end{eqnarray*}
A map that tells how to extend a function defined on
$A^{\ot 2}$ to a function defined on $(A\rtimes G)^{\ot 2}$ is
also from \cite{SW2}: 
\begin{eqnarray*}
\Theta_2^*:  \Hom_{\CC}(A^{\ot 2}, A \rtimes G)^G &\to & \Hom_{\CC}((A 
\rtimes G)^{\ot 2}, A \rtimes G)\\
\Theta_2^*(\kappa)(a_1 g_1 \ot a_2  g_2) & := & \kappa(a_1 \ot 
\lexp{g_1}{a_2})g_1g_2.
\end{eqnarray*}

We will identify $(A\rtimes G)\ot \Wedge^2_{{\bf q}^{-1}}(V^*)$ with 
$\Hom_{A^e}(A^e\ot \Wedge_{\bf q}^2(V), A\rtimes G)$
by sending $b\ot v_i^*\wedge v_j^*$ ($i<j$)
to the $A^e$-homomorphism taking $1\ot 1\ot v_k\wedge v_l$ ($k<l$)
to $b$ if $v_i\wedge v_j = v_k\wedge v_l$ and to 0 otherwise.
We will further identify $\Hom_{A^e}(A^e\ot U , W)$ with $\Hom_{\CC}(U,W)$,
for any vector space $U$ and $A^e$-module $W$, where convenient.

We will use the following, which is Theorem 4.3 of \cite{SW2}.
The hypothesis on the group action always holds in the case $q_{ij}=1$ for 
all $i,j$,
and as well in the case $q_{ij}=-1$ for all $i\neq j$.
For other choices of ${\bf q}$, the hypothesis is equivalent to conditions 
on the
entries of the matrices by which group elements act on $V$.

\begin{theorem}[\cite{SW2}]\label{thm: SW2}
Assume there is an action of $G$ on the quantum Koszul complex 
(\ref{label: free resolution}),
that is, an action of $G$ on each $A^e\ot \Wedge^i_{\bf q}(V)$ that 
commutes with
the differentials.
The composition $\Theta_2^* \R_2  \Psi_2^* $
induces an isomorphism
$$
\left(\bigoplus_{g\in G} \HH^{2}(A,A_g)\right)^G 
\xrightarrow{\sim} 
\HH^2(A \rtimes G).
$$
Moreover,  $\Theta_2^* \R_2  \Psi_2^* $ maps $\oplus_{g\in G} \HH^2(A, 
A_g)$
onto $\HH^2(A\rtimes G)$. 
\end{theorem}

For later use, we record a much simpler consequence.
Let $\alpha \in( A \rtimes G) \ot \Wedge_{{\bf q}^{-1}}^2(V^*)$. 
Then
\begin{equation}
\label{composition}
[\Theta_2^* \R_2   \Psi_2^* (\alpha)]
(v_i  \ot v_j ) = \frac{1}{|G|} \sum_{g \in G} \,  \lexp{g}\!
  {(\alpha (\Psi_2(1 \ot \lexp{g^{-1}}{v_i} \ot \lexp{g^{-1}}{v_j} \ot 
1)))}.
\end{equation}

Note that those elements of $\Hom_{A^e}(A^e\ot \Wedge_{\bf q}^2(V), 
A\rtimes G)^G
\cong ((A\rtimes G)\ot \Wedge_{{\bf q}^{-1}}^2(V^*))^G$ that correspond to 
{\em constant} Hochschild two-cocycles,
that is, those of degree $-2$ as maps from $(A\rtimes G)\ot (A\rtimes G)$ 
to $A\rtimes G$, 
are precisely
those in $(\CC G\ot\Wedge_{{\bf q}^{-1}}^2(V^*))^G$, due to the form of 
the chain map $\Psi_2$.
Thus we wish first to find those elements of 
$\CC G\ot \Wedge_{{\bf q}^{-1}}^2(V^*)$ that are in the kernel of $d_3^*$,
and then to restrict to $G$-invariants. 
Note that the intersection of the 
image of $d_2^*$ with $\CC G\ot\Wedge_{{\bf q}^{-1}}^2(V^*)$ is 0.
Applying our earlier formula, letting $\beta = [j] + [k]$, 
$$
\begin{aligned}
d_3^*( g\ot v_j^*\wedge v_k^*) \\
 & = \sum_{i\not\in\{j,k\}} (-1)^{\sum_{s=1}^i \beta_s} 
   \left[ \left(\left( \prod_{s=1}^i q_{s,i}^{\beta_s} \right) 
    v_i - \left(\prod _{s=i}^n q_{i,s}^{\beta_s}\right)
    \lexp{g}{v_i} \right)  g \right] \ot (v^*)^{\wedge(\beta + [i])} .
\end{aligned}
$$

\end{section}
\begin{section}{Quantum Drinfeld Hecke algebras and constant Hochschild 
2-cocycles}\label{LS-Theorem} 

Levandovskyy and Shepler \cite{LS} gave necessary
and sufficient conditions on the functions $\kappa_g$ 
for $\mH_{\mathbf{q}, \kappa}$ to be a quantum Drinfeld
Hecke algebra, and we restate their result as Theorem~\ref{thm: LS}  below.
Note that our formulation of their result is a little different as we
assume from the outset that $G$ acts by automorphisms on $S_{\bf q}(V)$,
our indices on ${\bf q}$ are reversed, and our functions $\kappa_g$ are
also reversed in the defining relations.
Theorem~\ref{thm: LS} will allow us to give explicitly the quantum
Drinfeld Hecke algebras as deformations of $S_{\bf q}(V)\rtimes G$
corresponding to Hochschild 2-cocycles found via the quantum 
Koszul resolution.
We will apply this theorem to particular types of group actions
in the next few sections.

For each group element $g\in G$, let $g^j_i$ be the
scalars for which 
$$
   \lexp{g}{v_j} = \sum_{i=1}^n g_i^j v_i.
$$
Define the {\bf quantum $(i,j,k,l)$-minor determinant} of $g$ as 
$$
    \ddet_{ijkl} (g) := g^j_lg^i_k - q_{ji} g^i_lg^j_k.
$$
It may be checked directly that for each $i,j$,
if $q_{ij}\neq 1$, then $g^i_kg^j_k=0$ for all $k$, since $G$ acts as
automorphisms on $S_{\bf q}(V)$.
(Apply $g$ to both sides of the equation $v_iv_j=q_{ij}v_jv_i$ and equate 
coefficients
of basis elements.) The following is a restatement of 
\cite[Theorem~7.6]{LS}. 

\begin{theorem}[\cite{LS}]\label{thm: LS} 
The algebra $\mH_{\mathbf{q}, \kappa}$ (defined in the introduction) is a 
quantum Drinfeld Hecke algebra
if and only if
\begin{itemize}
\item[(i)]  for all $g\in G$ and $1\leq i<j<k\leq n$, 
$$\hspace{.5cm} (q_{ki}q_{kj} \lexp{g}{v_k} - v_k)\kappa_g(v_j,v_i) + 
(q_{kj}v_j - q_{ji} \lexp{g}{v_j})
  \kappa_g(v_k,v_i) + (\lexp{g}{v_i} - q_{ji}q_{ki}v_i) \kappa_g(v_k,v_j) 
=0,$$
\item[(ii)] for all $i<j$ and all $g,h\in G$, 
   $ \  \kappa_{h^{-1}gh} (v_j,v_i) = \displaystyle{\sum_{k<l} 
\ddet_{ijkl}(h) \kappa_g(v_l,v_k)}$. 
\end{itemize}
\end{theorem}

We use Theorem \ref{thm: LS} to show in the next theorem that all constant 
Hochschild 2-cocycles
give rise to quantum Drinfeld Hecke algebras under an additional 
assumption.
First we need two lemmas, whose proofs are straightforward.

\begin{lemma}
The action of $G$ on $V$ extends to an action on $\Wedge_{\bf q}(V)$ by 
automorphisms
if, and only if, for all $g\in G$, $i\neq j$ and $k<l$,
$$
   (1-q_{ij}q_{lk} ) g^i_kg^j_l + (q_{ij} - q_{lk}) g^i_lg^j_k=0.
$$
\end{lemma}

Note that if all $q_{ij}=1$, or if $q_{ij}=-1$ for all $i\neq j$, the 
condition
in the lemma clearly holds. 
In general it imposes strong conditions on the matrix entries $g^i_j$:

\begin{lemma}\label{two-actions}
Assume that the action of $G$ on $V$ extends to an  action on $\Wedge_{\bf 
q}(V)$ 
by algebra automorphisms.
Then for all $g\in G$ and $i,j,k,l$ ($i<j$, $k<l$), if $g^i_lg^j_k\neq 0$ 
then
$q_{lk}=q_{ij}$, and if $g^i_kg^j_l \neq 0$, then $q_{lk}=q_{ij}^{-1}$.
\end{lemma}

\begin{theorem} \label{thm: constant}
Assume that the action of $G$ on $V$ extends to an action on $\Wedge_{\bf 
q}(V)$
by algebra automorphisms. Then  
each constant Hochschild 2-cocycle on $S_{\bf q}(V)\rtimes G$
gives rise to a quantum Drinfeld Hecke algebra.
\end{theorem}

\begin{proof}
Let $\alpha$ be any constant Hochschild 2-cocycle on $S_{\bf q}(V)\rtimes 
G$,
so that $\alpha$ may be expressed in terms of the quantum Koszul complex as
$$
  \alpha = \sum_{g\in G}\sum_{1\leq r<s\leq n} \alpha^g_{rs}g\ot 
v_r^*\wedge v_s^*
$$
for some scalars $\alpha^g_{rs}$. 
For each $i<j$ and $g\in G$, let $\kappa_g(v_i,v_j) = \alpha^g_{ij}$, so 
that 
$$
\sum_{g\in G} \kappa_g(v_i,v_j) g = \sum_{g\in G} \alpha^g_{ij} g 
= \alpha(v_i\ot v_j - q_{ij} v_j\ot v_i).
$$
We will check that the conditions of Theorem~\ref{thm: LS} hold,
and the conclusion will follow. 

For all $i<j<k$, $d_3^*(\alpha)$ applied to $1\ot 1\ot v_i\wedge v_j\wedge 
v_k$
yields
$$
\begin{aligned}
 & \alpha((v_i\ot 1 - q_{ij}q_{ik}\ot v_i)\ot v_j\wedge v_k 
    - (q_{ij}v_j\ot 1 - q_{jk}\ot v_j)\ot v_i\wedge v_k \\
 &\hspace{5cm}   + (q_{ik}q_{jk}v_k\ot 1 - 1\ot v_k)\ot v_i\wedge v_j)\\
 & = v_i\alpha(v_j \wedge v_k)-q_{ij}q_{ik}\alpha(v_j\wedge v_k) v_i
   -q_{ij}v_j\alpha(v_i\wedge v_k) + q_{jk}\alpha(v_i\wedge v_k) v_j\\
 &\hspace{5cm}   + q_{ik}q_{jk}v_k\alpha(v_i\wedge v_j) - \alpha(v_i\wedge 
v_j)v_k\\
 & =\sum_{g\in G} (\alpha^g_{jk}v_i g - q_{ij}q_{ik}\alpha^g_{jk}gv_i
   -q_{ij}v_j\alpha^g_{ik}g + q_{jk}\alpha^g_{ik}gv_j 
    + q_{ik}q_{jk}v_k\alpha^g_{ij}g -\alpha^g_{ij}gv_k).
\end{aligned}
$$
So $d_3^*(\alpha)=0$ if and only if
\begin{equation}\label{alphagjk}
  \alpha^g_{jk}(v_i-q_{ij}q_{ik} \lexp{g}{v_i}) 
-\alpha_{ik}^g(q_{ij}v_j-q_{jk} \lexp{g}{v_j})
   + \alpha^g_{ij}(q_{ik}q_{jk}v_k - \lexp{g}{v_k})=0
\end{equation}
for all $g\in G$, and $i,j,k$. This is indeed condition (i) of Theorem 
\ref{thm: LS}:
Replace $\kappa_g(v_j,v_i)$ by $-q_{ji}\kappa_g(v_i,v_j)$, and similarly 
for the others.
Then multiply the equation by $q_{ij}q_{ik}q_{jk}$. 

We claim that $G$-invariance of $\alpha$ is equivalent to condition (ii)
of Theorem \ref{thm: LS}: $\alpha$ is $G$-invariant if, and only if,
$\alpha(\lexp{h}{v_i} \wedge  \lexp{h}{v_j}) = \lexp{h}{(\alpha(v_i\wedge 
v_j))}$ for all $i,j$, and all $h\in G$.
Using the notation $h(v_i) = \sum_k h_k^i v_k$, we have 
\begin{eqnarray*}
  \alpha( \lexp{h}{v_i} \wedge  \lexp{h}{v_j}) & = & \sum_{k,l} h^i_kh^j_l 
\alpha(v_k\wedge v_l)\\
           &=& \sum_{k<l} h^i_kh^j_l\alpha(v_k\wedge v_l) - \sum_{k<l} 
q_{lk}
   h^i_lh^j_k \alpha(v_k\wedge v_l)\\
    &=& \sum_{k<l, \  g\in G} (h^i_kh^j_l - q_{lk}h^i_lh^j_k) 
\alpha^g_{kl} g
\end{eqnarray*}
and
$$
  \lexp{h}{(\alpha(v_i\wedge v_j))} = \lexp{h}{\left(\sum_{g\in G} 
\alpha^g_{ij}g\right)}
   = \sum_{g\in G} \alpha^g_{ij} hgh^{-1} = \sum_{g\in G} 
\alpha^{h^{-1}gh}_{ij} g .
$$
Equating the two, we find that
$$
  \alpha^{h^{-1}gh}_{ij} = \sum_{k<l} (h^i_kh^j_l-q_{lk}h^i_lh^j_k) 
\alpha^g_{kl}.
$$
By the proof of Lemma \ref{two-actions}, $q_{lk} h^i_lh^j_k = 
q_{ij}h^i_lh^j_k$, so we
may rewrite this as 
$$
  \alpha^{h^{-1}gh}_{ij} = \sum_{k<l} (h^i_kh^j_l - q_{ij}h^i_lh^j_k) 
\alpha^g_{kl}.
$$
Substituting $\alpha_{ij}^{h^{-1}gh} = - q_{ij}\alpha_{ji}^{h^{-1}gh}$, 
$\alpha_{kl}^g=-q_{kl} \alpha^g_{lk}$, $q_{kl} h^i_kh^j_l = 
q_{ij}h^i_kh^j_l$, and
$q_{kl}h^j_k h^i_l = q_{ji}h^j_kh^i_l$,
this is precisely Theorem \ref{thm: LS}(ii).
\end{proof}

Note that the hypothesis in Theorem \ref{thm: constant} is necessary
for the statement to make sense:
It is the action of $G$ on $\Wedge_{\bf q}(V)$ that allows cocycles
to be expressed in terms of the quantum Koszul complex, a {\em constant}
cocycle being defined via such an expression.

Combining Theorems~\ref{main-thm} and \ref{thm: constant}, we obtain the 
following.

\begin{theorem}
\label{thm: constant lifts}
Assume that the action of $G$ on $V$ extends to an action on $\Wedge_{\bf 
q}(V)$
by algebra automorphisms. Then each constant Hochschild 2-cocycle on 
$S_{\bf q}(V)\rtimes G$
lifts to a deformation of $S_{\bf q}(V)\rtimes G$ over $\CC[t]$. 
\end{theorem}

\end{section}
\begin{section}{Diagonal actions} \label{diagonal}

As before, let $G$ denote a finite group acting linearly on a vector space 
$V$.
In case $v_1,\ldots,v_n$ is a basis of common eigenvectors for $G$,
there is always an induced action of $G$ on $S_{\bf q}(V)$ 
and on $\Wedge_{\bf q}(V)$ by algebra automorphisms.
In this case, the Hochschild cohomology of $S_{\bf q}(V)\rtimes G$ 
was computed in 
\cite{NSW}, and we apply those results in this section: 
We give an explicit description of those elements of 
$\HH^2(S_{\bf q}(V) \rtimes G)$ corresponding to maps of degree~$-2$. 
Let $\lambda_{g,i}\in \CC$ be the scalars for which 
$\lexp{g}{v_i} = \lambda_{g,i}v_i$ for all 
$g \in G, i=1,\ldots,n$.

For each $g \in G$, define
\begin{equation}\label{Cg}
C_g := \left\{ \gamma \in (\N \cup \{-1\})^n \mid 
\text{ for each } i \in \{1, \ldots, n\}, \;
\prod_{s=1}^n q_{is}^{\gamma_s} = \lambda_{g,i} \text{ or } \gamma_i = -1 
\right\}.
\end{equation}

We recall the following from \cite{NSW}.

\begin{theorem}[\cite{NSW}]
If  $G$ acts diagonally on  $V$, then 
$\HH^{\bu}(S_{\bf q}(V),S_{\bf q}(V) \rtimes G)$ is the graded vector 
subspace of 
$(S_{\bf q}(V) \rtimes G) \ot \Wedge_{{\bf q}^{-1}}(V^*)$ given by:
$$
\HH^m(S_{\bf q}(V),S_{\bf q}(V) \rtimes G) 
\cong  \bigoplus_{g \in G}
\bigoplus_{\substack{\beta \in \O^n \\ |\beta| = m}} 
\bigoplus_{\substack{\alpha \in \N^n \\ \alpha - \beta \in C_g}}
\span_{\CC}\{(v^\alpha g) \ot {(v^*)}^{\wedge \beta}\},
$$
for all $m \in \N$, and $\HH^m(S_{\bf q}(V)\rtimes G)$ is its 
$G$-invariant subspace.
\end{theorem}

We immediately obtain the following. 

\begin{corollary}\label{cor: abelian}
The subspace of 
$\ds \HH^2(S_{\bf q}(V), S_{\bf q}(V) \rtimes G)$ consisting
of constant Hochschild 2-cocycles is isomorphic to 
$\ds \bigoplus_{g \in G} \bigoplus_{\substack{{r<s} \mathrm{\; s.t.} \\ 
\forall \, r' \neq r,s \\ q_{rr'}q_{sr'}=\lambda_{g,r'}}} 
\span_{\CC}\{g \ot v_r^* \wedge v_s^*\} . $
\end{corollary}


Let $\mathcal R$ denote a complete set of representatives of conjugacy 
classes in $G$, 
let $C_G(a)$ denote the centralizer of $a$ in $G$, and let $[G/C_G(a)]$ 
denote
a set of representatives of the cosets of $C_G(a)$ in $G$.
Combining Theorems \ref{thm: SW2}, \ref{thm: constant}, and 
Corollary~\ref{cor: abelian}, we obtain the following theorem. The 
notation $\delta_{i,j}$ is the Kronecker
delta. We note that the scalar $\lambda_{g,r}^{-1}\lambda_{g,s}^{-1}$ in 
the sum below
is independent of choice of representative $g$ of a coset of $C_G(a)$ under
the given assumption that $\lambda_{h,r}=\lambda_{h,s}^{-1}$ for all $h\in 
C_G(a)$.

\begin{theorem}
The maps $\kappa: V \times V \to \CC G$ for which $\mH_{{\bf q}, \kappa }$ 
is a quantum Drinfeld Hecke algebra form
a vector space with basis consisting of maps 
\[
f_{r,s,a}: V \times V \to \CC G: (v_i, v_j) \mapsto 
    (\delta_{i,r} \delta_{j, s} - q_{sr} \delta_{i,s} \delta_{j, r}) 
\sum_{g\in [G/C_G(a)]} \lambda_{g,r}^{-1}
        \lambda_{g,s}^{-1} \ gag^{-1}
\]
where $r < s$ and $a\in {\mathcal R}$ satisfy  
$ \ q_{rr'}q_{sr'}=\lambda_{a,r'}$ for all $r' \neq r,s$ and $\lambda_{h,r}
   = \lambda_{h,s}^{-1}$ for all $h\in C_G(a)$.
\end{theorem}

\begin{proof}
Let $\alpha = a \ot v_r^* \wedge v_s^*$ where $a \in G$ and $r<s$ satisfy 
$ \ q_{rr'}q_{sr'}=\lambda_{a,r'}$ for all $r' \neq r,s$. Thus, $\alpha$ 
represents a class in
$\HH^2(S_{\bf q}(V), S_{\bf q}(V) \rtimes G)$ that is {\em constant}, and 
by 
Corollary~\ref{cor: abelian}(ii), all constant Hochschild 2-cocycles have
this form. 
Applying \eqref{composition}, we get
\[
[\Theta_2^* \R_2   \Psi_2^* (\alpha)]
(v_i  \ot v_j ) = \frac{1}{|G|} \delta_{i,r} \delta_{j, s} \left(\sum_{h 
\in C_G(a)} \lambda_{h,r}\lambda_{h,s}\right)
	\sum_{g\in [G/C_G(a)]} \lambda_{g,r}^{-1} \lambda_{g,s}^{-1} \ 
gag^{-1}
\]
The sum $\sum_{h \in C_G(a)} \lambda_{h,r}\lambda_{h,s}$ is nonzero if and 
only if 
$\lambda_{h,r} = \lambda_{h,s}^{-1}$ for all $h\in C_G(a)$. To see this, 
define a 
homomorphism $\varphi:C_G(a) \to \CC^\times:h \mapsto 
\lambda_{h,r}\lambda_{h,s}$.
Since $\Im \varphi$ is finite, we have $\Im \varphi = \C_n$ for some $n$, 
where
$\C_n$ is the group of $n$th roots of unity. Thus,
\[
\sum_{h \in C_G(a)} \lambda_{h,r}\lambda_{h,s} = \sum_{h \in C_G(a)} 
\varphi(h) = |\ker \varphi| \sum_{\xi \in \C_n}\xi.
\]
The last expression is nonzero if and only if the group $\C_n$ is trivial, 
equivalently 
$\lambda_{h,r} = \lambda_{h,s}^{-1}$ for all $h\in C_G(a)$. The stated 
result now follows from
Theorems \ref{thm: SW2} and \ref{thm: constant}.
\end{proof}

One readily finds many examples of quantum Drinfeld Hecke algebras by
applying this theorem.
We illustrate next with one example in which the scalars $q_{ij}$ are not 
all 1 or $-1$. 

\begin{example}
Let $G$ be a cyclic group of order 3 generated by $g$.
Let $q$ be a primitive third root of 1.
Let $V = \CC^3$, with basis $v_1, v_2, v_3$, and let $q_{21}=q$,
$q_{32}=q$, and $q_{13}=q$. 
Let $G$ act on $V$ as follows: 
$$
    \lexp{g}{v_1} = q v_1, \ \ \ \lexp{g}{v_2} = q^2 v_2, \ \ \ 
\lexp{g}{v_3} = v_3.
$$
By the previous theorem, the map $\kappa:=f_{1,2,g}$ gives rise to a 
quantum
Drinfeld Hecke algebra, 
which is
$$ \mH_{\mathbf{q}, \kappa} = T(V)\rtimes G / (v_2v_1-qv_1v_2 + qg , \ 
      v_3v_2-qv_2v_3, \ v_1v_3-qv_3v_1).
$$
\end{example}

\end{section}
\begin{section}{Complex reflection groups: Natural representations} 
\label{natural}

In this section, we classify the quantum Drinfeld Hecke algebras for the 
complex reflection groups $G(m,p,n), \, n \geq 4$,
acting naturally on a vector space of dimension $n$. Specifically, we show 
that, with the exception of certain special cases,
there do not exist any nontrivial quantum Drinfeld Hecke algebras.

Let $m,p,n$ be positive integers with $p$ dividing $m$.
Let $\zeta$ denote a primitive $m$th root of unity and 
let 
$$
\zeta_i = \text{diag}(1, \ldots, \zeta, \ldots, 1)
$$
denote the diagonal matrix having $\zeta$ in the $i$th position.

Each element of $G(m,1,n)$ can be written uniquely in the form
$$
\zeta^\lambda \sigma \qquad \text{where } \lambda \in (\Z/m\Z)^n, \; 
\sigma \in G(1,1,n)\cong S_n,
$$
and
$$
\zeta^\lambda = \zeta_1^{\lambda_1} \zeta_2^{\lambda_2} \cdots 
\zeta_n^{\lambda_n}.
$$
The multiplication is determined by
$$
(\zeta^\lambda \sigma)(\zeta^\nu \tau) = \zeta^{\lambda + \sigma \cdot 
\nu}\sigma \tau,
$$
where $S_n$ acts on $(\Z/m\Z)^n$ by  permuting the coordinates, that is, 
$(\sigma \cdot \lambda)_i = \lambda_{\sigma^{-1}(i)}$.

The group $G(m,p,n)$ is defined to be the following subgroup of $G(m,1,n)$:
\[
G(m,p,n) = \{\zeta^\lambda \sigma \in G(m,1,n) \mid \lambda_1 + \lambda_2 
+ \cdots + \lambda_n \equiv_p 0\},
\]
where $ \equiv_p$ denotes equivalence modulo $p$.

Let $V$ be a vector space with an ordered basis $v_1,\ldots,v_n$. 
Let $S_\mathbf{-1}(V)$ be the algebra generated by $v_1,\ldots,v_n$ 
subject to the
relations $v_iv_j =-v_jv_i$ for all $i\neq j$.
In this case the corresponding quantum exterior algebra $\Wedge_{\bf 
{-1}}(V)$
is commutative: $v_i\wedge v_j=v_j\wedge v_i$ for all $i,j$.
Set $G:=G(m,p,n)$ and consider the natural action of $G$ on $V$. This 
action extends to
an action of $G$ on $S_\mathbf{-1}(V)$ by algebra automorphisms.

We begin with the following:

\begin{theorem}
\label{thm: constant cocycles}
Assume that $n \geq 4$.
The constant Hochschild 2-cocycles 
representing elements in $\HH^2(S_\mathbf{-1}(V), S_\mathbf{-1}(V)\rtimes 
G(m,p,n))$
form a vector space with basis all
$$
\begin{aligned}
 & \zeta^\lambda \ot v_r^*\wedge v_s^* \qquad 
(\lambda_{r'} \equiv_m 0 \text{ for all } {r'} \not \in \{r, s\} \text{ 
and } \lambda_r + \lambda_s \equiv_p 0),\\
 & \zeta^\lambda(rs)\ot v_r^*\wedge v_s^* \qquad 
(\lambda_{r'} \equiv_m 0 \text{ for all } {r'} \not \in \{r, s\} \text{ 
and } \lambda_r + \lambda_s \equiv_p 0),\\
 & \zeta^{\lambda_s} \left[\zeta^\lambda (rs)\right] \ot v_r^*\wedge 
v_{r'}^* + \zeta^\lambda (rs)\ot v_s^*\wedge v_{r'}^* \qquad 
(\lambda_{s'} \equiv_m 0 \text{ for all } {s'} \not \in \{r, s\} \text{ 
and } \lambda_r+\lambda_s \equiv_m 0), \\
 & \zeta^{\lambda_s}\left[\zeta^\lambda(rs)(r's')\right] \ot v_r^*\wedge 
v_{r'}^* 
  + \zeta^{\lambda_s+\lambda_{r'}}\left[\zeta^\lambda(rs)({r'}s')\right] 
\ot v_r^*\wedge v_{s'}^* \\
 &\hspace{1in}+ \zeta^\lambda(rs)(r's') \ot v_s^*\wedge v_{r'}^* 
  + \zeta^{\lambda_{r'}} \left[\zeta^\lambda(rs)(r's')\right] \ot 
v_s^*\wedge v_{s'}^*\\ 
 &  \hspace{1in} (\lambda_t \equiv_m 0 \text{ for all } t \not\in 
\{r,r',s,s'\} \text{ and } 
\lambda_r+\lambda_s \equiv_m 0 \equiv_m \lambda_{r'}+\lambda_{s'}),\\
 & \zeta^{\lambda_s+\lambda_{r'}} \left[\zeta^\lambda (rsr')\right] \ot 
v_r^*\wedge v_s^* + \zeta^\lambda (rsr') \ot v_s^*\wedge v_{r'}^* 
 + \zeta^{\lambda_s} \left[\zeta^\lambda (rsr')\right] \ot v_r^*\wedge 
v_{r'}^* \\
   &\hspace{1in} (\lambda_{s'} \equiv_m 0 \text{ for all } s'\not \in 
\{r,r',s\}
\text{ and } \lambda_r+\lambda_{r'}+\lambda_s \equiv_m 0).
\end{aligned}
$$
\end{theorem}

\begin{proof}
For each $g\in G$ and $1\leq r <s\leq n$, 
let $\alpha^g_{rs}\in {\mathbb{C}}$ and 
\begin{equation}\label{alpha}
  \alpha = \sum_{g\in G}\ \sum_{1\leq r<s\leq n}  \alpha_{rs}^g g \ot 
v_r^*\wedge v_s^*
\end{equation}
be an arbitrary constant cochain. (That is, $\alpha$ is a cochain whose  
polynomial part has degree 0 
in the factor
$S_{\bf q}(V)\rtimes G$ of the term $(S_{\bf q}(V)\rtimes G) \ot 
\Wedge^2_{{\bf q}^{-1}}(V^*)$
of the complex (\ref{new Hom(resolution) with G}).)
By the analysis in the proof of Theorem \ref{thm: constant},
$d_3^*(\alpha)=0$ if, and only if, equation (\ref{alphagjk}) holds, that 
is, since
$q_{rs} = -1$ for all $r\neq s$, 
\begin{equation}\label{three-terms}
  \alpha_{jk}^g (v_i-\lexp{g}{v_i}) + \alpha^g_{ik} (v_j - \lexp{g}{v_j})
   +\alpha_{ij}^g (v_k-\lexp{g}{v_k}) = 0
\end{equation}
for all $g\in G$ and all $i<j<k$. 
By symmetry, (\ref{three-terms}) holds for all triples $i,j,k$ if and only
if it holds for all $i<j<k$. 
If $\alpha$ is any of the types of cochains in the statement of the 
theorem,
these equations do indeed hold.

Conversely, assume that the equations (\ref{three-terms}) hold for a
cochain $\alpha$ of the form (\ref{alpha}). 
Fix a group element $g=\zeta^{\lambda}\sigma$.
Suppose $\alpha_{ij}^g\neq 0$ for some $i,j$.
If $\sigma =1$ or $\sigma = (ij)$, after reindexing, we obtain the first 
and second listed cocycles.
Now suppose there is some other $k$ for which $\sigma(k)\neq k$.
As a consequence of (\ref{three-terms}) we must have either $\sigma(k) =i$ 
or
$\sigma(k)=j$. Suppose $\sigma(k)=i$ (the other possibility is similar).
Then $\zeta^{\lambda_i} \alpha_{ij}^g = \alpha_{jk}^g$ and
either 
\begin{enumerate}
\item $\sigma(i)=k, \; \zeta^{\lambda_k} \alpha_{jk}^g = \alpha_{ij}^g$, 
and $\alpha^g_{ik}(v_j-\lexp{g}{v_j})=0$ or
\item $\sigma(i)=j, \; \sigma(j)=k, \; \zeta^{\lambda_j} \alpha^g_{jk} = 
\alpha^g_{ik}$,
and $\zeta^{\lambda_k}\alpha^g_{ik} = \alpha^g_{ij}$.
\end{enumerate}

In the second case, $\sigma$ contains the 3-cycle $(ijk)$ in its
disjoint cycle decomposition. We claim that $\sigma = (ijk)$:
If not then there is some other $l$ with $\sigma(l)\neq l$.
Apply (\ref{three-terms}) to $i,j,l$ to obtain
$$
  \alpha_{jl}^g (v_i-\zeta^{\lambda_i}v_j) + 
\alpha^g_{il}(v_j-\zeta^{\lambda_j}v_k)
   +\alpha^g_{ij}(v_l -\lexp{g}{v_l}) =0.
$$
This is not possible since $\alpha^g_{ij}(v_l - \lexp{g}{v_l}) \neq 0$.
Putting all the conditions together and reindexing we obtain the fifth 
listed cocycle.

In the first case, $\sigma$ contains the transposition $(ik)$ in its 
disjoint
cycle decomposition. If $\sigma = (ik)$, then, after
reindexing, we obtain a third type or a second type plus a third type of 
the cocycles listed.
Now suppose there is some other $l$ for which $\sigma(l)\neq l$.
The equation  (\ref{three-terms}) applied to $i,j,l$ becomes
$$
  \alpha^g_{jl}(v_i-\zeta^{\lambda_k}v_k) + 
\alpha^g_{il}(v_j-\lexp{g}{v_j})
  +\alpha^g_{ij}(v_l-\lexp{g}{v_l}) =0.
$$
This forces $\sigma(j)=l, \; \sigma(l)=j, \; \alpha_{jl}^g=0, \; 
\alpha_{ik}^g=0, \; \zeta^{\lambda_l}\alpha_{il}^g=\alpha_{ij}^g$,
and $\zeta^{\lambda_j}\alpha_{ij}^g=\alpha_{il}^g$. 
We now have the product of transpositions $(ik)(jl)$ as part of the
disjoint cycle decomposition of $\sigma$.
We claim that $\sigma=(ik)(jl)$: If there were another $m$ for which
$\sigma(m)\neq m$, applying the above analysis to the triple $i,j,m$
would force $\sigma(j) = m$ as well, a contradiction. 
Applying (\ref{three-terms}) to $i,k,l$ gives
$$
  \alpha^g_{kl}(v_j-\zeta^{\lambda_l}v_l) + 
\alpha^g_{jk}(v_l-\zeta^{\lambda_j}v_j)=0.
$$
This forces $\zeta^{\lambda_j}\alpha_{jk}^g=\alpha_{kl}^g$ and 
$\zeta^{\lambda_l}\alpha_{kl}^g=\alpha_{jk}^g$. 
Putting all the conditions together and reindexing we obtain the fourth 
listed cocycle.
\end{proof}

\begin{lemma}
\label{lemma: image}
Assume that $n \geq 4$. Let $\alpha$ be a 2-cocycle from the list in 
Theorem~\ref{thm: constant cocycles}.
The image of $v_i \ot v_j \; (i \neq j)$ under $(\Theta_2^* \R_2  
\Psi_2^*)(\alpha)$ is zero whenever
\begin{enumerate}
\item[(i)] $\alpha$ is of the first, third, or fourth type and $m \geq 2$, 
or
\item[(ii)] $\alpha$ is of the second or fifth type and $m \geq 3$.
\end{enumerate}
\end{lemma}
\begin{proof}
Let $\alpha = \zeta^{\lambda_s} \left[\zeta^\lambda (rs)\right] \ot 
v_r^*\wedge v_{r'}^* + \zeta^\lambda (rs)\ot v_s^*\wedge v_{r'}^*$,
the third 2-cocycle from the list in Theorem~\ref{thm: constant cocycles}.
By \eqref{composition}, $|G|$ times $\left[(\Theta_2^* \R_2  
\Psi_2^*)(\alpha)\right](v_i \ot v_j)$ is equal to
\[
\sum_{\tau \in S_n} \sum_{\substack{\nu \in (\Z/m\Z)^n \\ \nu_1 + \cdots + 
\nu_n \equiv_p 0}}
\lexp{\zeta^{\nu}\tau}{\left( \alpha(\Psi_2(1 \ot 
\lexp{{(\zeta^{\nu}\tau)}^{-1}}{v_i} \ot 
\lexp{{(\zeta^{\nu}\tau)}^{-1}}{v_j}\ot 1))\right)}.
\]
Applying  (\ref{psi-two}) and then evaluating $\alpha$, the above 
expression becomes
\[
\sum_{\substack{\nu \in (\Z/m\Z)^n \\ \nu_1 + \cdots + \nu_n \equiv_p 0}} 
\left( 
\sum_{\substack{\tau \in S_n \\ \tau(r)=i \\ \tau(r')=j}} \zeta^{\lambda_s 
- \nu_i - \nu_j} 
\left[\lexp{\zeta^{\nu}\tau}{\left(\zeta^{\lambda}(rs)\right)}\right]
+\sum_{\substack{\tau \in S_n \\ \tau(s)=i \\ \tau(r')=j}} \zeta^{- \nu_i 
- \nu_j} \left[\lexp{\zeta^{\nu}\tau}{\left(\zeta^{\lambda}(rs)
\right)}\right] \right).
\]
Note that the map $\tau \mapsto \tau(rs)$ is a bijection between the 
two sets over which the two inner summations are taken.
Taking this fact and the conjugation action into account, 
the above expression can be rewritten as
\[
\begin{aligned}
&\sum_{\substack{\tau \in S_n \\ \tau(r)=i \\ \tau(r')=j}} 
\sum_{\substack{\nu \in (\Z/m\Z)^n \\ \nu_1 + \cdots + \nu_n \equiv_p 0}}
\left(
\zeta^{\lambda_s - \nu_i - \nu_j} \left[\zeta^{\nu + \tau \cdot \lambda - 
(i\tau(s)) \cdot \nu}(i\tau(s))\right]
+\zeta^{- \nu_i - \nu_j} \left[\zeta^{\nu + \tau (rs) \cdot \lambda - 
(i\tau(s)) \cdot \nu}(i\tau(s))\right] 
\right)\\
=&\sum_{k \not = i,j} \sum_{\substack{\tau \in S_n \\ \tau(r)=i \\ 
\tau(r')=j \\ \tau(s)=k}} 
\sum_{\substack{\nu \in (\Z/m\Z)^n \\ \nu_1 + \cdots + \nu_n \equiv_p 0}}
\left(
\zeta^{\lambda_s - \nu_i - \nu_j} \left[\zeta^{\nu + \tau \cdot \lambda - 
(ik) \cdot \nu}(ik)\right]
+\zeta^{- \nu_i - \nu_j} \left[\zeta^{\nu + \tau (rs) \cdot \lambda - (ik) 
\cdot \nu}(ik)\right] 
\right)\\
=&\sum_{k \not = i,j} \sum_{\substack{\tau \in S_n \\ \tau(r)=i \\ 
\tau(r')=j \\ \tau(s)=k}} 
\sum_{\substack{\nu \in (\Z/m\Z)^n \\ \nu_1 + \cdots + \nu_n \equiv_p 0}}
\left(
\zeta^{\lambda_s - \nu_i - \nu_j} \left[\zeta_i^{\lambda_r + \nu_i - 
\nu_k}\zeta_k^{\lambda_s + \nu_k - \nu_i }(ik)\right]
+\zeta^{- \nu_i - \nu_j} \left[\zeta_i^{\lambda_s + \nu_i - \nu_k 
}\zeta_k^{\lambda_r + \nu_k - \nu_i }(ik)\right] 
\right).
\end{aligned}
\]
Since the inner-most summand does not depend on $\tau$, the above 
expression is proportional to
\[
(*) \qquad \sum_{k \not = i,j}  \sum_{\substack{\nu \in (\Z/m\Z)^n \\ 
\nu_1 + \cdots + \nu_n \equiv_p 0}}
\left(
\zeta^{\lambda_s - \nu_i - \nu_j} \left[\zeta_i^{\lambda_r + \nu_i - 
\nu_k}\zeta_k^{\lambda_s + \nu_k - \nu_i }(ik)\right]
+\zeta^{- \nu_i - \nu_j} \left[\zeta_i^{\lambda_s + \nu_i - \nu_k 
}\zeta_k^{\lambda_r + \nu_k - \nu_i }(ik)\right] 
\right).
\]
Rewriting the second summation in $(*)$ as 
$\ds \sum_{\nu_1, \ldots, \hat{\nu}_l, \ldots, \nu_n \in \Z/m\Z} 
\sum_{\substack{\nu_l \in \Z/m\Z \\ \nu_1 + \cdots + \nu_n \equiv_p 0}}$,
where $l$ is chosen so that $l \neq i,j,k$, and observing that the 
inner-most summand of $(*)$ depends only on 
$\nu_i, \; \nu_j,$ and $\nu_k$, we see that the expression in $(*)$ is 
proportional to
\[
\begin{aligned}
&\sum_{k \not = i,j} \sum_{l_1,l_2,l_3 \in \Z/m\Z}
\left(
\zeta^{\lambda_s - l_1 - l_2} \left[\zeta_i^{\lambda_r + l_1-l_3} 
\zeta_k^{\lambda_s + l_3-l_1} (ik)\right]
+\zeta^{- l_1 - l_2} \left[\zeta_i^{\lambda_s + l_1-l_3} 
\zeta_k^{\lambda_r +l_3-l_1} (ik)\right]
\right)\\
=&\sum_{k \not = i,j} \sum_{l_1,l_2\in \Z/m\Z}
\left(\zeta^{\lambda_s - l_1 - l_2} 
\sum_{l_3 \in \Z/m\Z} \left[ \zeta_i^{\lambda_r + l_3} \zeta_k^{\lambda_s 
-l_3} (ik)\right]
+\zeta^{- l_1 - l_2}\sum_{l_3 \in \Z/m\Z} \left[\zeta_i^{\lambda_s + l_3} 
\zeta_k^{\lambda_r-l_3} (ik)\right] \right)\\
=&\sum_{k \not = i,j} 
\left(
T_1
\sum_{l_3 \in \Z/m\Z} \left[\zeta_i^{\lambda_r + l_3} \zeta_k^{\lambda_s 
-l_3} (ik)\right] 
+ T_2 \sum_{l_3 \in \Z/m\Z} \left[\zeta_i^{\lambda_s +l_3} 
\zeta_k^{\lambda_r-l_3} (ik)\right] \right),
\end{aligned}
\]
where $T_1=\sum_{l_1,l_2\in \Z/m\Z} \zeta^{\lambda_s - l_1 - l_2}$ and
$T_2=\sum_{l_1,l_2\in \Z/m\Z} \zeta^{-l_1 - l_2}$. The above expression is 
zero since both $T_1$ and $T_2$ are 
zero whenever $m \geq 2$.

The proofs for the other 2-cocycles are similar.
\end{proof}

\begin{remark}
The calculations corresponding to the second and fifth 2-cocycles involve 
the sum $\ds \sum_{l \in \Z/m\Z} \zeta^{2l}$,
which is zero if and only if $m \geq 3$, explaining the assumption $m \geq 
3$ in the second part of the previous lemma.
\end{remark}

Combining Theorems \ref{thm: SW2}, \ref{thm: constant}, \ref{thm: constant 
cocycles}, and Lemma \ref{lemma: image}, we obtain:

\begin{theorem}
If $n \geq 4$ and $m \geq 3$, then the vector space of maps $\kappa: V 
\times V \to \CC G(m,p,n)$ for which
$\mH_{{\bf q}, \kappa }$ is a quantum Drinfeld Hecke algebra is trivial, 
and hence there are no nontrivial 
quantum Drinfeld Hecke algebras.  
\end{theorem}

\begin{remark}
The analogue of the above theorem for the case ${\mathbf q}={\mathbf 1}$ 
(that is, $q_{ij}=1$ for all $i,j$)
 was proved by Ram and Shepler \cite{RS}.
\end{remark}

Next, we consider the special cases $G(1,1,n)$, $G(2,1,n)$, and 
$G(2,2,n)$. In all three
cases we assume that $n \geq 4$. In these cases, there do exist nontrivial 
quantum Drinfeld Hecke algebras and
in what follows we classify them.

\medskip
\noindent {\bf The symmetric group $G(1,1,n)=S_n$}
\medskip

In this case, Theorem~\ref{thm: constant cocycles} reduces to the 
following:

\begin{theorem}
\label{thm: constant cocycles m=1}
Assume that $n \geq 4$. The constant Hochschild 2-cocycles 
representing elements in $\HH^2(S_\mathbf{-1}(V), S_\mathbf{-1}(V)\rtimes 
S_n)$
form a vector space with basis all
$$
\begin{aligned}
 & 1\ot v_r^*\wedge v_s^*,\\
 & (rs)\ot v_r^*\wedge v_s^*, \\
 & (rs)\ot (v_r^*\wedge v_{r'}^* + v_s^*\wedge v_{r'}^*),\\
 & (rs)(r's') \ot (v_r^*\wedge v_{r'}^* + v_r^*\wedge v_{s'}^* + 
v_s^*\wedge v_{r'}^* 
               + v_s^*\wedge v_{s'}^*),\\
 & (rs{r'})\ot (v_r^*\wedge v_s^* + v_s^*\wedge v_{r'}^* + v_r^*\wedge 
v_{r'}^*).
\end{aligned}
$$
\end{theorem}



Combining Theorems \ref{thm: SW2}, \ref{thm: constant}, and \ref{thm: 
constant cocycles m=1}, we obtain:

\begin{theorem}
Assume that $n \geq 4$.
The maps $\kappa:V \times V \to \CC S_n$ for which $\mH_{{\bf q}, \kappa 
}$ is a quantum Drinfeld Hecke algebra form
a five-dimensional vector space with basis consisting of maps $V \times V 
\to \CC S_n$ determined by their 
effect on pairs $(v_i,v_j)$, $i \neq j$, given by
\[
\begin{aligned}
&  (v_i,v_j) \mapsto 1,\\
&  (v_i,v_j) \mapsto (ij),\\
&  (v_i,v_j) \mapsto \sum_{k\neq i,j} ((ik) + (jk)),\\
&  (v_i,v_j) \mapsto \sum_{k,l\not\in\{i,j\}} (ik)(jl),\\
&  (v_i,v_j) \mapsto \sum_{k\neq i,j} ((ijk) + (ikj)).
\end{aligned}
\]
\end{theorem}

\medskip
\noindent {\bf The hyperoctahedral group $G(2,1,n)=WB_n$}
\medskip

We know from Lemma~\ref{lemma: image} that if $\alpha$ is a first, third, 
or fourth type of $2$-cocycle in 
the list in Theorem~\ref{thm: constant cocycles}, then the image of $v_i 
\ot v_j \; (i \neq j)$ under 
$\left[(\Theta_2^* \R_2  \Psi_2^*)(\alpha)\right]$ is zero. For the second 
and fifth type of $2$-cocycles, we have
the following:

\begin{lemma}\label{m2}
If $\alpha$ is a second or fifth type of $2$-cocycle, then the 
corresponding image of $v_i \ot v_j \; (i \neq j)$ 
under $\left[(\Theta_2^* \R_2  \Psi_2^*)(\alpha)\right]$ is, respectively, 
proportional to
\[
\begin{aligned}
 & 
\begin{cases}
(ij)-\zeta_i\zeta_j(ij) & \text{if } \lambda_r, \lambda_s \text{ have the 
same parity }\\
\zeta_i(ij)-\zeta_j(ij) & \text{if } \lambda_r, \lambda_s \text{ have 
different parity },
\end{cases}  \\ 
 & 
\ds \sum_{k \not = i,j} 
(
2(ijk) - 2\zeta_i\zeta_j(ijk) - 2\zeta_j\zeta_k(ijk) + 
2\zeta_i\zeta_k(ijk) 
+ (ikj) - \zeta_i\zeta_j(ikj) + \zeta_j\zeta_k(ikj) - \zeta_i\zeta_k(ikj)
).
\end{aligned}
\]
\end{lemma}

Combining Theorems \ref{thm: SW2}, \ref{thm: constant}, \ref{thm: constant 
cocycles}, and Lemma~\ref{m2}, we obtain:

\begin{theorem} \label{m2 classification}
Assume that $n \geq 4$.
The maps $\kappa:V \times V \to \CC G(2,1,n)$ for which $\mH_{{\bf q}, 
\kappa }$ is a quantum Drinfeld Hecke algebra form
a two-dimensional vector space with basis consisting of maps $V \times V 
\to \CC G(2,1,n)$ determined by their 
effect on pairs $(v_i,v_j)$, $i \neq j$, given by
\[
\begin{aligned}
&  (v_i,v_j) \mapsto (ij)-\zeta_i\zeta_j(ij),\\
&  (v_i,v_j) \mapsto \ds \sum_{k \not = i,j} 
((ijk) - \zeta_i\zeta_j(ijk) - \zeta_j\zeta_k(ijk) + \zeta_i\zeta_k(ijk) \\
& \hspace{1.5in} + (ikj) - \zeta_i\zeta_j(ikj) + \zeta_j\zeta_k(ikj) - 
\zeta_i\zeta_k(ikj)).
\end{aligned}
\]
\end{theorem}

\medskip
\noindent {\bf The type $D_n$ Weyl group $G(2,2,n)=WD_n$}
\medskip

As was the case for $G(2,1,n)$, we do not get nontrivial quantum Drinfeld 
Hecke algebras from
the first, third, or fourth type of $2$-cocycles,
but there do exist nontrivial quantum Drinfeld 
Hecke algebras for the second and fifth 
type of $2$-cocycles.
The corresponding result for this case is simply 
Theorem~\ref{m2 classification} with 
$G(2,1,n)$ replaced by $G(2,2,n)$.

\end{section}
\begin{section}{Complex reflection groups: Symplectic representations} 
\label{symplectic}

In this section, we classify the quantum Drinfeld Hecke algebras for 
certain symplectic representations of 
the complex reflection groups $G(m,1,n), \, m \text{ even}, \, \, n \geq 
3$. We also explain how our classification
includes all of the braided Cherednik algebras of Bazlov and Berenstein 
\cite{BB}.

Let $m$ be an {\em even} positive integer, let $G$ denote the complex 
reflection group $G(m,1,n)$, and let $\zeta$ denote a 
primitive $m$th root of unity.
Let $U$ be a vector space of dimension $n$, and let $x_1,\ldots x_n$ be an 
ordered basis for $U$. Denote by $y_1, \ldots, y_n$ the basis for $U^*$ 
dual to $\{x_i\}$, so that 
$U^*=\span_{\CC}\{y_1,\ldots y_n\}$.
Set $v_i := x_i, v_{n+i} := y_i, 1 \leq i \leq n$, and $V:=U \oplus U^*$, 
so that $\dim V = 2n$ and $v_1, \ldots, v_{2n}$ is an ordered basis for 
$V$.

Consider the natural action of $G$ on $U$. 
The dual action of $G$ on $U^*$ is determined by
\[  
\lexp{g}{y_i} = \varepsilon^{-1} y_j \text{ if and only if } \lexp{g}{x_i} 
= \varepsilon x_j.
\] 
Thus, we have an action of $G$ on $V:=U \oplus U^*$, and this action 
induces an action of $G$
on  $S_{\minusone}(V)$ by algebra automorphisms, where  $\minusone$ is 
defined below.

For each pair $i,j$ of elements in $\{1,\ldots,2n\}$, let $q_{ij}$ be a 
scalar defined as follows:
$$
q_{ij} := 
\begin{cases}
1, & \text{if~} i \equiv j\mod n \\
-1, & \text{otherwise}.
\end{cases}
$$
Denote by $\minusone$ the tuple of scalars $(q_{ij})$.

Let $\C := \{\zeta^{l} \mid l \in \Z/m\Z\}$. Since $|\C|=m$ is even,  $-1$
is an element of $\C$.
For each $\varepsilon \in \C$ and $i,j \in \{1, \ldots n\}, i \not = j$,
define $\sigma_{ij}^{(\varepsilon)} \in \GL(U)$ by 
$$
\sigma_{ij}^{(\varepsilon)}(x_k) := 
\begin{cases}
x_k & \text{if~} k \not \in \{i,j\} \\
\varepsilon x_j & \text{if~} k=i \\
-\varepsilon^{-1} x_i & \text{if~} k=j.
\end{cases}
$$
That is, $\sigma_{ij}^{(\varepsilon)} = (-\zeta_i^{-l}) \zeta_j^l (ij) \in 
G(m,1,n)$,
where $\varepsilon = \zeta^l$.

For each $\varepsilon \in \C$ and $i \in \{1, \ldots n\}$,
define $t_i^{(\varepsilon)} \in \GL(U)$ by 
$$
t_i^{(\varepsilon)}(x_k) := 
\begin{cases}
x_k & \text{if~} k \not = i \\
\varepsilon x_i & \text{if~} k=i.
\end{cases}
$$
That is, $t_i^{(\varepsilon)} = \zeta_i^{l} \in G(m,1,n)$,
where $\varepsilon = \zeta^l$.

\begin{theorem}
\label{thm: constant cocycles'}
Suppose $m$ is even and $n \geq 3$. The constant Hochschild 2-cocycles 
representing elements in $\HH^2(S_\mathbf{-1}(V), S_\mathbf{-1}(V)\rtimes 
G)$
form a vector space with basis all
$$
\begin{aligned}
 & t_r^{(-1)}t_s^{(-1)}\ot x_r^*\wedge x_s^*,\\
 & t_r^{(-1)}t_s^{(-1)}\ot y_r^*\wedge y_s^*,\\
 & t_r^{(-1)}t_s^{(-1)}\ot x_r^*\wedge y_s^* \qquad (r \neq s),\\
 & t_r^{(\varepsilon)}\ot x_r^*\wedge y_r^*,\\
 & \sigma_{rs}^{(\xi)} \ot x_r^* \wedge y_r^* + \sigma_{rs}^{(\xi)} \ot 
x_s^* \wedge y_s^*
   + \xi \sigma_{rs}^{(\xi)} \ot x_r^* \wedge y_s^* -  \xi^{-1} 
\sigma_{rs}^{(\xi)} \ot x_s^* \wedge y_r^* 
   \qquad (r \neq s).\\
\end{aligned}
$$
\end{theorem}

\begin{proof}
For each $g\in G(m,1,n)$ and $1\leq r<s\leq 2n$, let $\alpha^g_{rs}\in\CC$ 
and
$$
   \alpha = \sum_{g\in G(m,1,n)} \ \sum_{1\leq r<s\leq 2n}
     \alpha^g_{rs}g \ot v_r^*\wedge v_s^*
$$
be an arbitrary constant cochain. Then $d_3^*(\alpha)=0$
if, and only if, equation (\ref{alphagjk}) holds for all $g$, $i$, $j$, 
$k$.
Fix $g = \zeta ^{\lambda}\sigma$ and suppose $\alpha^g_{ij}\neq 0$ for 
some $i,j$.

First we consider the case $\sigma =1$, so that
$g=\zeta^{\lambda}$ for some $\lambda$, and $g$ acts diagonally on the 
basis
$x_1,\ldots,x_n,y_1,\ldots, y_n$. 
Assume that $1\leq i<j\leq n$. 
In this case, since $\alpha^g_{ij}\neq 0$, as a consequence
of (\ref{alphagjk}) we have $\zeta^{\lambda_k} = q_{ik}q_{jk}$ for each 
$k\not\in\{i,j\}$,
$1\leq k\leq n$, and $\zeta^{-\lambda_k} = q_{ik}q_{jk}$ when $n+1\leq 
k\leq 2n$.
By choosing $k=i+n$ or $k=j+n$ we obtain $\zeta^{-\lambda_i}=-1$, 
$\zeta^{-\lambda_j}=-1$.
Other choices yield $\zeta^{\lambda_k}=1$ when $k\not\in \{i,j\}$.
Thus we obtain the first cocycle in the list. A similar analysis yields the
second cocycle. If we instead assume $1\leq i\leq n$ and $n+1\leq j\leq 
2n$, $j\neq i+n$,
we obtain the third. Finally assume $1\leq i\leq n$ and $j=i+n$. Then for 
all
$k\neq i $ ($1\leq k\leq n$), 
$$
  \zeta^{\lambda_k} = q_{ik}q_{i+n,k} =1
$$
and for all $k\neq i+n$ ($n+1\leq k\leq 2n$),
$$
   \zeta^{-\lambda_k}=q_{ik}q_{i+n,k}=1,
$$
while there is no restriction on $\lambda_i$.
Thus we obtain the fourth cocycle in the list,  completing the
analysis of the case $\sigma=1$.

Next consider the case 
$g=\zeta^{\lambda}\sigma$, $\sigma\neq 1$, and $\alpha_{ij}^g\neq 0$
for some $i,j$. Assume $1\leq i<j\leq n$ and $\sigma = (ij)$.
Consider equation (\ref{alphagjk}) in case $k=i+n$:
$$
  \alpha_{j,i+n}^g(v_i - q_{ij}q_{i,i+n} \zeta^{\lambda_j}v_j)
   - \alpha^g_{i,i+n}(q_{ij}v_j - q_{j,i+n} \zeta^{\lambda_i} v_i)
   + \alpha^g_{ij} (q_{i,i+n} q_{j,i+n} v_{i+n} - \zeta^{-\lambda_j}
     v_{j+n}) =0.
$$
Since $\alpha_{ij}^g\neq 0$, we must have
$$
   q_{i,i+n} q_{j,i+n} v_{i+n} - \zeta^{-\lambda_j} v_{j+n} = 0,
$$
  which is not possible. Therefore this case cannot occur. 
Similarly if $n+1\leq i<j\leq 2n$ and $\sigma = (i-n, j-n)$, we arrive
at a contradiction.
Assume now that $1\leq i\leq n$, $n+1\leq j\leq 2n$,
and $\sigma = (i, j-n)$. A similar analysis forces $\alpha^g_{j-n, i+n}$, 
$\alpha^g_{j-n, j}$,
and $\alpha^g_{i,i+n}$ all to be nonzero and one obtains the fifth cocycle
in the list.

Finally assume that $\sigma \neq 1$, $\alpha^g_{ij}\neq 0$ for some $i,j$, 
and $\sigma$
is {\em not} a 2-cycle moving only $i,j$. 
Then there is some $k\not\in \{ i , j \}$, $1\leq k\leq n$,
for which $\sigma(k)\neq k$. Assume first that $1\leq i<j\leq n$.
Applying equation (\ref{alphagjk}) to $i,j,k$,
we see that $\sigma(k)$ is forced  to be either $i $ or $j$.
Without loss of generality assume $\sigma(k) = i$. Then 
$\zeta^{\lambda_i}\alpha_{ij}^g
= \alpha^g_{jk}$, and again equation (\ref{alphagjk}) forces either one of 
the
following two possibilities:
\begin{itemize}
\item[(1)] $\sigma(i)=k$, $\alpha_{jk}^g q_{ij} \zeta^{\lambda_k}
     = \alpha^g_{ij} q_{jk}$, and $\alpha^g_{ik}(q_{ij} v_j - q_{jk}
     \lexp{g}{v_j}) =0$, or
\item[(2)] $\sigma(i)=j$, $\sigma(j)=k$, $\alpha^g_{jk} q_{ik} 
\zeta^{\lambda_j}
    =\alpha_{ik}^g$, and $\alpha_{ik}^g \zeta^{\lambda_k} = 
\alpha^g_{ij}q_{ik}$.
\end{itemize}
In case (1), further consider equation (\ref{alphagjk}) with $k$ replaced 
by
$i + n$:
$$
  \alpha^g_{j,i + n} (v_i - q_{ij}q_{i,i+n} \zeta^{\lambda_k} v_k)
   - \alpha^g_{i,i+ n} (q_{ij} v_j - q_{j,i+n} \lexp{g}{v_j}) 
      + \alpha_{ij}^g (q_{i,i+n} q_{j,i+n} v_{i+ n} - \zeta^{-\lambda_k}
    v_{k+ n} ) = 0 .
$$
This forces $\alpha_{j,i+ n}^g=0$ since $v_i,v_k$ cannot both occur in 
other
terms of this equation. However, since also $j\not\in \{ i+n, k+n\}$ and 
$\alpha_{ij}^g\neq 0$,
the equation cannot hold. Therefore this case does not occur. 
A similar analysis applies if we assume $n+1\leq i<j\leq 2n$ or if $1\leq 
i\leq n$
and $n+1\leq j\leq 2n$.

In case (2), consider equation (\ref{alphagjk}) applied to indices $i,j,i+ 
n$:
$$
  \alpha_{j,i+ n}^g (v_i - q_{ij} q_{i,i+ n} \zeta^{\lambda_j} v_j)
  - \alpha^g_{i,i+ n} (q_{ij} v_j -q_{j,i+ n} \zeta^{\lambda_k}v_k)
   + \alpha_{ij}^g (q_{i,i+ n} q_{j,i+ n} v_{i+ n} - \zeta^{- \lambda_j}
  v_{j+ n})=0.
$$
Since $\alpha_{ij}^g\neq 0$ and $v_{i+ n}$, $v_{j+ n}$ do not occur in 
other
terms in this equation, the equation in fact does not hold. Therefore
there is no such cocycle. A similar analysis applies if we assume
$n+1\leq i<j\leq 2n$ or if $1\leq i\leq n$ and $n+1\leq j\leq 2n$.
\end{proof}


\begin{lemma} \label{m2'}
For each 2-cocycle $\alpha$ from the list in Theorem~\ref{thm: constant 
cocycles'},
the image of $v_i \ot v_j \; (i \neq j)$ under $\left[(\Theta_2^* \R_2  
\Psi_2^*)(\alpha)\right]$ is, respectively, given by
$$
\begin{aligned}
 & 0 \\
 & 0 \\
 & 0 \\
 & \begin{cases}
   x_i \ot x_j &\mapsto \; 0\\
   y_i \ot y_j &\mapsto \; 0\\
   y_j \ot x_i &\mapsto \; 0\\
   x_i \ot y_j &\mapsto \; 0 \qquad (i \neq j)\\
   x_i \ot y_i &\mapsto \; \frac{1}{n}t_i^{(\varepsilon)}
   \end{cases} \\
 & \begin{cases}
   x_i \ot x_j &\mapsto \; 0\\
   y_i \ot y_j &\mapsto \; 0\\
   y_j \ot x_i &\mapsto \; 0\\
   x_i \ot y_j &\mapsto \; \ds \frac{2}{mn(n-1)} \sum_{\eta \in \C} 
\eta \sigma_{ij}^{(\eta)} \qquad (i \neq j)\\
   x_i \ot y_i &\mapsto \; \ds \frac{2}{mn(n-1)} 
\sum_{\substack{\eta \in \C \\ j \neq i}} 
\sigma_{ij}^{(\eta)}
   \end{cases}
\end{aligned}
$$
\end{lemma}


Combining Theorems \ref{thm: SW2}, \ref{thm: constant}, \ref{thm: constant 
cocycles'}, and Lemma~\ref{m2'}, we obtain:

\begin{theorem} \label{kappas}
Assume that $m$ is even and $n \geq 3$.
The maps $\kappa:V \times V \to \CC G(m,1,n)$ for which $\mH_{{\bf q}, 
\kappa }$ is a quantum Drinfeld Hecke algebra form
an $(m+1)$-dimensional vector space with basis consisting of maps $V 
\times V \to \CC G(m,1,n)$ determined by their 
effect on pairs $(x_i,x_j), (y_i,y_j), (x_i,y_j), (x_i,y_i)$, $i \neq j$, 
given by
\[
\begin{aligned}
 &\begin{cases}
  (x_i,x_j) \mapsto 0\\
  (y_i,y_j) \mapsto 0\\
  (x_i,y_j) \mapsto 0\\
  (x_i,y_i) \mapsto t_i^{(\varepsilon)}
  \end{cases}\\
 &\begin{cases}
  (x_i,x_j) \mapsto 0 \\
  (y_i,y_j) \mapsto 0\\
  (x_i,y_j) \mapsto \ds  \sum_{\eta \in \C} \eta 
\sigma_{ij}^{(\eta)} \\
  (x_i,y_i) \mapsto \ds \sum_{\substack{\eta \in \C \\ j \neq i}} 
\sigma_{ij}^{(\eta)}.
 \end{cases}
\end{aligned}
\]
\end{theorem}

\begin{remark}
Let $\C'$ be a subgroup of $\C$, the group 
of $m$th roots of unity ($m$ even), and let
$W_{\C,\C'}$ denote the subgroup of $G(m,1,n)$ generated by all 
$\sigma_{ij}^{(\varepsilon)}, \, \varepsilon \in \C$
and $t_i^{(\varepsilon')}, \, \varepsilon' \in \C'$. Let $c:\C' \to \CC$ 
be any function, and
$c_{\varepsilon'}:= c(\varepsilon ')$.
In \cite{BB}, Bazlov and Berenstein defined algebras 
$\mathcal{H}_c(W_{\C,\C'})$ with generators $x_1,x_2,\ldots x_n, 
y_1,y_2,\ldots y_n$
subject to the following relations
\begin{enumerate}
\item[(i)] $x_ix_j + x_jx_i = y_iy_j + y_jy_i = 0$ for all $i \not = j$;
\item[(ii)] $gx_ig^{-1} = \lexp{g}{x_i}, \, gy_ig^{-1} = \lexp{g}{y_i}$ 
for all $g \in W_{\C,\C'},\, i = 1, \ldots, n$;
\item[(iii)] $\ds y_jx_i + x_iy_j = c_1 \sum_{\varepsilon \in \C} 
\varepsilon \sigma_{ij}^{(\varepsilon)}$ for all $i \not = j$, and
\item[] $\ds y_ix_i - x_iy_i = 1 + c_1 \sum_{\substack{\varepsilon \in \C 
\\ j \neq i}} \sigma_{ij}^{(\varepsilon)}
         + \sum_{\varepsilon' \in \C' \setminus \{1\}} c_{\varepsilon'} 
t_i^{(\varepsilon')}$ for $i = 1, \ldots, n$.
\end{enumerate}

Let $f_\varepsilon, \,  \varepsilon \in \C$, denote the maps of the first 
type in Theorem~\ref{kappas} and let
$\tilde{f}$ denote the map of the second type in Theorem~\ref{kappas}.
The algebras $\mathcal{H}_c(W_{\C,\C'})$ are precisely the quantum 
Drinfeld Hecke algebras arising from the map
$
 \left(f_1 + \sum_{\varepsilon' \in \C'\backslash \{1\}} \left(
c_{\varepsilon'} f_{\varepsilon'} + c_1\tilde{f}\right)\right):V \times V 
\to \CC W_{\C,\C'}.
$

\end{remark}

\end{section}



\end{document}